\documentclass[a4paper,11pt,oneside]{article}
\usepackage[margin=3cm]{geometry}

\usepackage{latexsym,amsmath,amsthm,amssymb,calc}
\usepackage[utf8]{inputenc}
\usepackage{dsfont}
\usepackage{xcolor}
\usepackage{graphicx}
\usepackage{mdwlist}
\usepackage{enumerate}
\usepackage{setspace}
\usepackage{cite}

\newtheorem{thrm}{Theorem}[section]

\newtheorem{prpstn}[thrm]{Proposition}

\theoremstyle{definition}

\newtheorem{assumptions}[thrm]{Assumptions}

\theoremstyle{remark}
\newtheorem{rmrk}[thrm]{Remark}

\numberwithin{equation}{section}

\usepackage{algorithm}
\usepackage{algorithmicx}
\usepackage{algpseudocode}
\algrenewcommand\algorithmicrequire{\makebox[32pt][l]{\textrm{input}}}
\algrenewcommand\algorithmicensure{\makebox[32pt][l]{\textrm{output}}}
\algrenewcommand\algorithmicfunction{\textrm{function}}
\algrenewcommand\algorithmicwhile{\textrm{while}}
\algrenewcommand\algorithmicdo{}
\algrenewcommand\algorithmicend{\textrm{end}}
\algrenewcommand\algorithmicforall{\textrm{for all}}
\algrenewcommand\algorithmicfor{\textrm{for}}
\algrenewcommand\algorithmicrepeat{\textrm{repeat}}
\algrenewcommand\algorithmicuntil{\textrm{until}}

\DeclareFontFamily{U}{mathx}{\hyphenchar\font45}
\DeclareFontShape{U}{mathx}{m}{n}{
      <5> <6> <7> <8> <9> <10>
      <10.95> <12> <14.4> <17.28> <20.74> <24.88>
      mathx10
      }{}
\DeclareSymbolFont{mathx}{U}{mathx}{m}{n}
\DeclareMathSymbol{\bigtimes}{1}{mathx}{"91}

\newcommand{\N}{\mathds{N}}
\newcommand{\R}{\mathds{R}}
\newcommand{\Z}{\mathds{Z}}

\DeclareMathOperator{\supp}{supp}

\DeclareMathOperator*{\argmin}{arg\,min}

\DeclareMathOperator{\range}{range}
\DeclareMathOperator{\ops}{ops}

\newcommand{\id}{{\rm id}}

\newcommand{\spl}[1]{{\rm \ell}_{#1}}

\newcommand{\spH}[1]{{\rm H}^{#1}}
\newcommand{\spL}[1]{{\rm L}_{#1}}

\newcommand{\Hcal}{{\mathcal{H}}}
\newcommand{\Rcal}{{\mathcal{R}}}

\newcommand{\Ical}{{\mathcal{I}}}

\newcommand{\Acal}{{\mathcal{A}}}

\newcommand{\Ncal}{{\mathcal{N}}}
\newcommand{\Rank}{\Rcal}

\DeclareMathOperator{\rank}{rank}

\DeclareMathOperator{\apply}{\textsc{apply}}
\DeclareMathOperator{\coarsen}{\textsc{coarsen}}
\DeclareMathOperator{\rhs}{\textsc{rhs}}
\DeclareMathOperator{\recompress}{\textsc{recompress}}
\DeclareMathOperator{\solve}{\textsc{solve}}

\providecommand{\abs}[1]{\lvert#1\rvert}
\providecommand{\bigabs}[1]{\bigl\lvert#1\bigr\rvert}

\providecommand{\biggabs}[1]{\biggl\lvert#1\biggr\rvert}

\providecommand{\norm}[1]{\lVert#1\rVert}

\providecommand{\ceil}[1]{\lceil#1\rceil}

\newcommand{\UU}{\mathbb{U}}

\newcommand{\bu}{{\bf u}}
\newcommand{\bU}{{\bf U}}

\newcommand{\bv}{{\bf v}}
\newcommand{\bw}{{\bf w}}

\newcommand{\cA}{{\mathcal{A}}}
\newcommand{\ga}{\gamma}
\newcommand{\garatio}{{\rho_\ga}}

\newcommand{\As}{{\Acal^s}}

\newcommand{\AH}[1]{{\Acal_\Hcal({#1})}}

\newcommand{\kk}[1]{{\mathsf{#1}}}
\newcommand{\rr}[1]{{\mathsf{#1}}}

\newcommand{\KK}[1]{{\mathsf{K}_{#1}}}

\newcommand{\Psvd}[2]{\operatorname{P}_{\UU({#1}),{#2}}}
\newcommand{\hatPsvd}[1]{\operatorname{\hat P}_{#1}}

\newcommand{\rsvd}{\operatorname{r}}

\newcommand{\hatCctr}[1]{\operatorname{\hat C}_{#1}}
\newcommand{\Restr}[1]{\operatorname{R}_{#1}} 

\newcommand{\constsvd}{\kappa_{\rm P}}
\newcommand{\constcrs}{\kappa_{\rm C}}

\newcommand{\bA}{\mathbf{A}}
\newcommand{\bT}{\mathbf{T}}
\newcommand{\bbf}{\mathbf{f}}
\newcommand{\bg}{\mathbf{g}}
\newcommand{\cL}{{\mathcal{L}}}
\newcommand{\cD}{{\mathcal{D}}}

\newcommand{\cN}{{\mathcal{N}}}

\newcommand{\hatbS}{\mathbf{\hat S}}
\newcommand{\bS}{\mathbf{S}}

\newcommand{\tbT}{{\tilde{\mathbf{T}}}}

\newcommand{\hdimtree}[1]{\mathcal{D}_{#1}}
\newcommand{\hroot}[1]{{0_{#1}}}
\newcommand{\leaf}[1]{\cL(\hdimtree{#1})}
\newcommand{\nonleaf}[1]{\cN(\hdimtree{#1})}

\newcommand{\leftchild}{{{\rm c}_1}}
\newcommand{\rightchild}{{{\rm c}_2}}
\newcommand{\child}[1]{{{\rm c}_{#1}}}
\newcommand{\hsum}[1]{\mathrm{\Sigma}_{#1}}

\newcommand{\dd}{{\rm rank}}

\newcommand{\om}[1]{{\omega_{#1}}}
\newcommand{\omi}[2]{{\hat\omega_{#1,#2}}}

\newcommand{\Sc}{{\bS}}

\newcommand{\Sci}[1]{{\hat\bS_{#1}}}

\newcommand{\Pc}{{\mathbf{P}}}
\newcommand{\bC}{{\mathbf{C}}}
\newcommand{\omin}{\omega_\mathrm{min}}
\DeclareMathOperator{\diag}{diag}

\makeatletter
\def\blfootnote{\xdef\@thefnmark{}\@footnotetext}
\makeatother

\title{Adaptive Low-Rank Methods for Problems on Sobolev Spaces with Error Control in $\spL{2}$}

\author{M.\ Bachmayr and W.\ Dahmen}

\begin{document}

\maketitle

\blfootnote{This work has been supported in part by the DFG SFB-Transregio 40, the Excellence Initiative of the German Federal and State Governments
 (RWTH Aachen  Distinguished Professorship), NSF grant DMS 1222390, and the ERC advanced grant BREAD.} 

\begin{abstract}
Low-rank tensor methods for the approximate solution of second-order elliptic partial differential equations in high dimensions have recently attracted significant attention. A critical issue is to rigorously bound the error of such approximations, not with respect to a fixed finite dimensional discrete background problem, but with respect to the exact solution of the continuous problem. While the energy norm offers a natural error measure corresponding to the underlying operator considered as an isomorphism from the energy space onto its dual, this norm requires a careful treatment in its interplay with the tensor structure of the problem. In this paper we build on our previous work on energy norm-convergent
subspace-based tensor schemes contriving, however, a modified formulation  which now enforces convergence only in $\spL{2}$.
In order to still be able to exploit the mapping properties of elliptic operators, a crucial ingredient of our approach is 
the development and analysis of a suitable asymmetric preconditioning scheme. We provide estimates for the computational complexity of the resulting method in terms of the solution error and study the practical performance of the scheme in numerical experiments. In both regards, we find that controlling solution errors in this weaker norm leads to substantial simplifications and to a reduction of the actual numerical work required for a certain error tolerance.


\end{abstract}
 
\section{Introduction}\label{sec:intro}
For a given open product domain $\Omega = \Omega_1 \times \cdots \times \Omega_d \subset \R^d$,  we are interested in
approximately  solving problems of the form
\begin{equation}
\label{1.1}
- {\rm div}( M \nabla u)= f \;\text{in $\Omega$},\quad u|_{\partial\Omega} =0,
\end{equation}
where $M$ is a symmetric uniformly positive definite $(d\times d)$-matrix over $\Omega$. Here, we are interested in
the {\em spatially high-dimensional} regime $d\gg 1$. We always assume that $M$ is sparse and thus causes only a weak coupling of the variables $x_i\in \Omega_i$.
As guiding examples one may think of  $M= {\rm diag}(M_1,\ldots,M_d)$, where
the $M_i$ are sufficiently benign functions of $x_i$ only, or of a constant tridiagonal matrix $M$. 

For simplicity of exposition, we deliberately keep \eqref{1.1} on the level of a specific model problem.
However, what follows applies in essence also to natural variants of \eqref{1.1}, for instance, when the Dirichlet boundary conditions are replaced
(partially or throughout) by Neumann conditions, as long as the type of boundary condition remains the same on
each $(d-1)$-face of $\Omega$. While above the $\Omega_i$ are intervals, one could also consider a product of $d$ more general low-dimensional domains. %

For product domains and weakly coupling diffusion matrices, the differential 
operator in \eqref{1.1} has formally low rank, that is, its action only leads to a moderate increase in the ranks of suitable tensor representations. This justifies the hope that, for instance, for separable right hand sides
the solution may be approximable efficiently by low-rank tensor expansions, where the low-dimensional factors in the rank-one summands are not predefined basis functions but are allowed to depend on the solution. This hope is indeed supported by
substantial numerical evidence and, at least for diagonal $M$, also on a theoretical level \cite{DDGS}.

The numerical treatment of such solution-dependent basis functions still necessitates their expansion in terms of suitable low-dimensional reference basis functions. In this combination of low-rank representations and basis expansions of corresponding tensor components, we are thus in fact dealing with two levels of approximation. Between these, a proper balance needs to be maintained in the convergence to the exact solution, since both allowing large tensor ranks for coarse discretizations and using very fine discretizations with inaccurate low-rank approximations will, especially at high accuracies, lead to excessive numerical costs. 
Common strategies for low-rank approximations start from a fixed discretizations and use tensor representations as a linear algebra tool, see e.g.\ \cite{BNZ:14,Kressner:11,Ballani:13}.
Such a fixed discretization corresponds to a fixed finite reference basis for representing the tensor components. In this setting, however, one cannot address the necessary intertwining of subspace approximation and adaptive refinement of tensor factor representations. We thus need to deal with several closely connected issues: obtaining a posteriori error information that can guide the adaptive refinement of the reference basis, ensuring that the underlying representation of the differential operator does not become ill-conditioned as the basis is refined, and avoiding inappropriately large tensor ranks.

These considerations have motivated the approach put forward in \cite{B,BD} on a general level and in \cite{BD2}
with special focus on problems of the form \eqref{1.1}. A central idea there is that rigorous a posteriori error bounds driving convergent approximation schemes should rely on a faithful approximation to the {\em residual of the continuous problem}
which, in turn, should reflect the accuracy of the approximate solution. This seems to be possible only when exploiting the mapping properties
of the operator $A$ induced by the classical weak formulation
\begin{equation}
\label{1.2}
\langle A u,v\rangle := \int_\Omega M\nabla u\cdot \nabla v \,dx = \langle f,v\rangle,\quad v\in \spH{1}_0(\Omega),
\end{equation}
over the space $\spH{1}_0(\Omega)$.
In fact, denoting by $\underline{c}, \bar C$ the smallest and largest eigenvalue of $M$
one has 
\begin{equation}
\label{mappinH1}
\|A\|_{\spH{1}_0(\Omega)\to \spH{-1}(\Omega)}\leq \bar C,\quad \|A^{-1}\|_{H^{-1}(\Omega)\to \spH{1}_0(\Omega)}\le \underline{c}^{-1}\,,
\end{equation}
which is equivalent to saying that errors  in the $\spH{1}$-norm can be faithfully estimated by residuals in the dual norm $\|\cdot\|_{\spH{-1}(\Omega)}$.

It is unfortunately not entirely straightforward to exploit these facts for a rigorous error control of low-rank tensor approximations. Subspace-based tensor formats, whose stability properties play an important role in devising reliable computational routines, are not immediately amenable to spaces that are not endowed with cross-norms, see \cite{BD2} for a detailed discussion.
Therefore, the strategy in \cite{B,BD,BD2} is based on transforming the problem first to an equivalent one where 
the transformed operator $\bA$ is an isomorphism mapping an $\spl{2}$-space over an infinite product index set, which {\em is} a space endowed with a cross-norm, {\em onto itself}. This transformation requires a Riesz basis for the
energy space $\spH{1}_0(\Omega)$. A suitable basis of this type can be obtained by rescaling an orthonormal tensor product  wavelet basis of $\spL{2}(\Omega)$. Unfortunately, and this is the price to be paid, the rescaling destroys separability of the basis functions
and, as a consequence, causes the resulting operator representation $\bA$ to have {\em infinite rank}. Aside from the role of suitable recompression and coarsening
operators given in \cite{BD}, a key ingredient in still constructing 
low-rank approximations with controlled energy norm accuracy for elliptic problems are {\em adaptive finite-rank rescaling operators}
proposed and analyzed in \cite{BD2}. They are based on new specially tailored relative error bounds for
exponential sum approximations to the function $g(t)=t^{-1/2}$. This ultimately led to an adaptive refinement scheme
generating approximate solutions represented in hierarchical tensor formats, convergent in energy norm with 
near-optimal complexity, for each fixed spatial dimension $d$, with respect to ranks and representation sparsity of the tensor factors \cite{BD2}.

Nevertheless, the fact that the energy norm is not a cross-norm and the resulting unbounded tensor ranks of the representation $\bA$ significantly impede the control of rank growth in the iterates. The central question addressed in the present work is therefore:

\begin{quote}Can one devise a solver that provides approximate solutions in hierarchical tensor format at a significantly lower numerical cost by enforcing convergence only in a norm that is \emph{weaker} than the energy norm, namely $\|\cdot\|_{\spL{2}(\Omega)}$?
\end{quote}

Of course, there is no hope of avoiding the above mentioned \emph{``scaling problem''} completely. In one way or another, a rigorous convergence analysis has
to make use of the mapping properties of the underlying operator, which always refers to a pair of spaces  of which is at least one is {\em not} endowed 
with a cross-norm. However, if one has {\em  full elliptic regularity}, the underlying operator is also an isomorphism from
$\spH{2}(\Omega)\cap \spH{1}_0(\Omega)\to \spL{2}(\Omega)$ and, by duality, also from $\spL{2}(\Omega)$ onto 
$(\spH{2}(\Omega)\cap \spH{1}_0(\Omega))'$. Since $\Omega$ as a Cartesian product of open intervals (or more generally of {\em convex} low-dimensional domains) is convex this is indeed the case. 

Adhering to the basic idea in \cite{BD,BD2} of transforming the variational problem first into an equivalent problem over
the space $\ell_2(\nabla^d)$, with $\nabla$ a countable index set, of order-$d$ tensors---so as to be able to employ {\em subspace-based tensor formats}---we now need to contrive an {\em asymmetric} preconditioner to arrive at an ideal convergent iteration for the infinite-dimensional problem on $\ell_2(\nabla^d)$.
 This central issue is addressed in Section \ref{sec:asymprecond}. Section \ref{sec:analysis} is devoted to the precise formulation of the new algorithm and its
 convergence and complexity analysis. Finally, in Section \ref{sec:numre}, the theoretical findings are illustrated and quantified by numerical experiments.
 
 We close this section with recalling some primarily technical preliminaries from \cite{BD,BD2}, where a more self-contained exposition can be found.

\subsection{Prerequisites }\label{ssec:1.2}
\subsubsection{Tensor Representations}
For simplicity of exposition, in what follows we focus as in \cite{BD2} on problems of the form \eqref{1.1} with a constant diffusion matrix $M$ and $\Omega =(0,1)^d$ so that the operator 
\begin{equation}\label{eq:example_operator}
 A  u := -\sum_{i,j=1}^d m_{ij} \partial_i\partial_j u \,,
\end{equation}
with constant coefficients $m_{ij}$ and symmetric positive  definite $M=(m_{ij})\in \R^{d\times d}$ satisfies \eqref{mappinH1}.  Furthermore, to avoid certain technicalities, we impose the slightly stronger assumption that $M$ is \emph{diagonally dominant}. 

In order to transform \eqref{1.1} into an equivalent problem over sequence spaces, we  employ
a tensor product {\em wavelet basis} 
$$
  \bigl\{\Psi_\nu := \psi_{\nu_1} \otimes\cdots\otimes\psi_{\nu_d}\colon \nu =(\nu_1,\ldots,\nu_d) \in\nabla^d \bigr\}\,,
$$
where $\{ \psi_\nu \}_{\nu\in\nabla}$ is an orthonormal basis of $\spL{2}(0,1)$ and $\{ 2^{-2\abs{\nu}} \psi_\nu \}_{\nu\in\nabla}$ is a Riesz basis of 
$\spH{2}(0,1)\cap \spH{1}_0(0,1)$. Note that this requires, in particular, that the wavelets vanish on $\partial\Omega$, that is, the univariate factor
wavelets satisfy $\psi_\nu(0)=\psi_\nu(1)=0$, $\nu\in\nabla$.
The corresponding wavelet representation of $A$ is then given by the infinite matrix
\begin{equation}
\label{eq:Tdef}
\bT := \bigl(  \langle \Psi_\mu, A \Psi_\nu \rangle \bigr)_{\mu,\nu\in\nabla^d} .
\end{equation}
Since the homogeneous boundary conditions are built into the basis $\{ \Psi_\nu \}$, finding the solution $u$ of \eqref{1.1}
is equivalent to finding its wavelet coefficient sequence 
\begin{equation}
\label{ucoeff}
\bu:= (u_\nu)_{\nu\in\nabla^d}, \quad u_\nu := \langle u,\Psi_\nu\rangle := \int_\Omega u\, \Psi_\nu\, dx.
\end{equation}
Defining $\bg = (\langle f,\Psi_\nu\rangle )_{\nu\in\nabla^d}$, the sequence $\bu$, in turn, is the solution of 
\begin{equation}
\label{Tsystem}
\bT \bu = \bg.
\end{equation}
Thus our objective is to solve \eqref{Tsystem}. Note that the operator $\bT$ is unbounded as an operator from $\ell_2(\nabla^d)$ to itself,
where as usual $\ell_2(\nabla^d)$ is the space of square summable sequences over the index set $\nabla^d$ endowed with the norm
$$
\|\bv\|:= \|\bv\|_{\ell_2(\nabla^d)}:= \Big(\sum_{\nu\in\ell_2(\nabla^d)}|v_\nu|^2\Big)^{1/2}. 
$$

 For the moment we postpone the discussion of the choice of subspace of $\ell_2(\nabla^d)$ for which \eqref{Tsystem} is  supposed to hold
and explain first some algebraic features of \eqref{Tsystem}. Since $\nabla^d$ is a product set, we view any element $\bv\in \ell_2(\nabla^d)$
as {\em a tensor of order} $d$. As an operator acting on such tensors, $\bT$ has finite rank. More precisely, as has been pointed
out in \cite{BD2}, $\bT$ has
 the tensor representation
\begin{equation} 
\label{eq:unrescaled_tuckersum}
 \bT = \sum_{1\leq n_1,\ldots,n_d\leq R} c_{n_1,\ldots,n_d} \bigotimes_i \bT^{(i)}_{n_i} \,,
\end{equation}
with $R=4$, and
\begin{align} 
\label{T1}
  \bT^{(i)}_1 := \bT_1 &= \bigl( \langle \psi_\nu, \psi_\mu \rangle \bigr)_{\mu,\nu\in\nabla} = \id\,, &  
 \bT^{(i)}_2 := \bT_2 &:= \bigl( \langle \psi'_\nu, \psi'_\mu \rangle \bigr)_{\mu,\nu\in\nabla}  \,,\\
 \label{T2}
 \bT^{(i)}_3 := \bT_3 &:= \bigl( \langle \psi_\nu, \psi'_\mu \rangle \bigr)_{\mu,\nu\in\nabla}\,, &
 \bT^{(i)}_4 := \bT_4 &:= \bigl( \langle \psi'_\nu, \psi_\mu \rangle \bigr)_{\mu,\nu\in\nabla}\,.
\end{align}
Here the nonzero entries $c_{n_1,\ldots,n_d}$ of the sparse coefficient tensor $\mathbf{c}$ are given by
$c_{2,1,\ldots,1} = m_{11}$, $c_{1, 2,1,\ldots,1} = m_{22}$, \ldots, $c_{3,4,1,\ldots,1} = c_{4,3,1,\ldots,1} = m_{12}$, \ldots, $c_{3,1,4,1,\ldots,1} = c_{4,1,3,1,\ldots,1} = m_{13}$, and so forth. Thus, for $A$ as in \eqref{eq:example_operator}, in general we have that $R=4$.
Note further that, due to the 
 homogeneous Dirichlet boundary conditions, integration by parts shows $\bT_3 = -\bT_4$, which gives
\begin{equation}
\label{simplifying}
 \bT_3 \otimes \bT_4 = \bT_4 \otimes \bT_3 = - \bT_3\otimes \bT_3 ,
\end{equation}
and thus a reduction to $R=3$. 

We shall now introduce some basic notions of tensor representations. For further details and references, we refer to \cite{Hackbusch:12}.
The particular representation format for the operator $\bT$ chosen in \eqref{eq:unrescaled_tuckersum} corresponds to the so-called {\em Tucker format} for tensors of order $d$.
Accordingly, as mentioned earlier, regarding $\bu$ as a tensor of order $d$ on $\nabla^d =
\bigtimes_{i=1}^d \nabla$, it can be represented in terms of the Tucker format
\begin{equation}\label{eq:tuckerd}
  \mathbf{u} =  \sum_{k_1=1}^{r_1} \cdots \sum_{k_d=1}^{r_d}
    a_{k_1,\ldots,k_d} \,\mathbf{U}^{(1)}_{k_1} \otimes \cdots \otimes
    \mathbf{U}^{(d)}_{k_d} \, \,,
\end{equation}
where $\mathbf{a}= (a_{k_1,\ldots,k_d})_{1\leq k_i\leq r_i:i=1,\ldots,d}$ is called
the \emph{core tensor} and each  matrix $\mathbf{U}^{(i)}= \big(\bU^{(i)}_{\nu_i,  k_i }\big)_{\nu_i\in \nabla^{d_i},1\leq k_i\leq r_i}$ with
orthonormal column vectors $\mathbf{U}^{(i)}_k\in\spl{2}(\nabla^{d_i})$, $k = 1,\ldots,r_i$,
is called the $i$-th \emph{orthonormal mode frame} (here we admit $r_i=\infty, i=1,\ldots,d$).
We refer to \cite{BD,BD2} for the precise definitions and notation, to which we will   adhere in this paper as well.   Of course, for an operator $\bT$ on $\ell_2(\nabla^d)$ and an element $\bu\in \ell_2(\nabla^d)$ that are both given as
representations in the Tucker format, the image $\bT\bu$ can readily be expressed in the Tucker format by a combination of the core tensors and the application of the 
$\bT^{(i)}_k$ to the mode frames $\bU^{(i)}_{k_i}$, see \cite{BD} for details.

Since the core tensor $\mathbf{a}$ in  \eqref{eq:tuckerd} still depends on $d$ indices, for large $d$ it will generally have far too many entries for a direct representation.
For this reason, we focus in what follows on the 
{\em hierarchical Tucker format} \cite{Hackbusch:09-1}, which is obtained by further decomposing $\mathbf{a}$ into successive compositions of
third-order tensors as
$$   
\mathbf{a}= \Bigl( \hsum{\hdimtree{d}}\bigl(\{
\mathbf{B}^{(\alpha,k)}\}\bigr) \Bigr)_{(k_\beta)_{\beta\in\mathcal{L}(\hdimtree{d})}}
\\
:= \sum_{(k_\gamma)_{\gamma\in 
    \Ical(\hdimtree{d})} }
       \prod_{\delta\in \nonleaf{d}}
      B^{(\delta,k_\delta)}_{(k_{\leftchild(\delta)},k_{\rightchild(\delta)})}
      \,.        
$$
This is based on a fixed {\em binary dimension tree} $\cD_d$ obtained by successive bisections of the set of coordinate 
indices $\hroot{d} := \{1,\ldots,d\}$,
which forms the root node. 
Moreover,  %
singletons $\{ i\} \in \hdimtree{d}$ are referred to as \emph{leaves},  and elements
of $\hdimtree{d}\setminus\bigl\{\hroot{d},\{1\},\ldots,\{d\} \bigr\}$ as \emph{interior nodes}.
The set of leaves is denoted by $\leaf{d}$, where we
additionally set $\nonleaf{d} :=
\hdimtree{d}\setminus\mathcal{L}(\mathcal{D}_d)$. 
The functions
$$
\child{i}: \cD_d\setminus \cL(\cD_d)\to \cD_d\setminus \{\hroot{d}\}, 
\qquad i=1,2\,,
$$
produce the ``left'' and ``right'' children $\child{i}(\alpha)\subset \alpha$ of a non-leaf node $\alpha \in \mathcal{N}(\cD_d)$.

With each node $\alpha\in \cD_d$ we associate the \emph{matricization}
$T^{(\alpha)}_\bu$ of $\bu$, obtained by rearranging the entries of the tensor into an infinite matrix representation of a Hilbert-Schmidt operator using the indices in $\nabla^\alpha$ as row indices. The dimensions of the ranges of these
operators yield the \emph{hierarchical ranks}
$\dd_\alpha(\bu) := \dim \range T^{(\alpha)}_\bu$ for $\alpha \in \cD_d$. Except for $\alpha=\hroot{d}$, where we always have $\dd_{\hroot{d}}(\bu)=1$, these are collected in the \emph{hierarchical rank vector} $
\rank (\bu) =  {\rank_{\hdimtree{d}}(\bu)} := (\rank_\alpha (\bu) )_{\alpha\in\hdimtree{d}\setminus\{\hroot{d}\}} 
$
and give rise to the hierarchical tensor classes
\begin{equation*}
  \Hcal(\rr{r}) := \bigl\{ \mathbf{u} \in
  \spl{2}(\nabla^d) \colon  \rank_\alpha(\mathbf{u}) \leq {r}_\alpha \text{
  for all $\alpha \in \hdimtree{d}\setminus\{\hroot{d}\}$} \bigr\} \,.
\end{equation*}
For singletons $\{i\}\in \hdimtree{d}$, we briefly write $\rank_i (\bu) := \rank_{\{i\}}(\bu)$. We denote by $\mathcal{R} \subset (\N_0 \cup\{\infty\})^{\hdimtree{d}\setminus\{\hroot{d}\}}$ the set of hierarchical rank vectors for which $ \Hcal(\rr{r})$ is nonempty.

Again, there is an analogous hierarchical format for operators, i.e., the core tensor $\mathbf{c}$ in \eqref{eq:unrescaled_tuckersum}
is further decomposed as a product of tensors of order three, and the format is consistent when applying an operator to a tensor, see \cite{BD}. The hierarchical ranks in the representation of $\mathbf{c}$ will be denoted by $R_\alpha$, $\alpha \in \hdimtree{d}$.
In what follows we are mainly interested in two scenarios, namely that $M= {\rm diag}\,(m_{ii})_{i=1}^d$
or that $M$ is tridiagonal. In the former case, we have $R=2$ as well as $R_\alpha \leq 2$. For tridiagonal $M$, in general one obtains $R=4$ and $R_\alpha \leq 5$, but in the present case, due to \eqref{simplifying}, this reduces to $R=3$ and $R_\alpha\leq 4$. We refer to \cite[Example 3.2]{BD2}
for more details.
\subsubsection{Recompression and Coarsening}\label{sssec:1.2.2}

The basic strategy suggested in \cite{BD,BD2}, which we follow here as well, is to solve \eqref{Tsystem} iteratively.
At a first glance this looks promising since the application of the finite-rank operator $\bT$ to a finite-rank iterate 
produces (at least for a suitably truncated finite-rank right hand side) a new iterate of finite rank. 
As mentioned earlier, at least two principal obstructions arise. 
First, the action of the operator as well as the summation of finite rank tensors increase the tensor ranks in each step, so that a straightforward
iteration would give rise to exponentially increasing ranks. Second, increasing the ranks of tensor expansions has to go hand in hand with growing
the supports of increasingly more accurately resolved mode frames. In this section we briefly recall from \cite{BD,BD2} how to deal with these issues. The key point is to devise suitable tensor recompression and coarsening schemes
that automatically find near-best approximations from the classes $\Hcal(\rr{r})$ whose mode frames have    near-minimal supports. Again we refer to \cite{BD} 
for a detailed derivation and recall here the main results for later use.
The {\em hierarchical singular value decomposition} ($\Hcal$SVD) (cf.\ \cite{Grasedyck:10})
allows one to identify for given tensor $\bv$ a system of mode frames, denoted by
$\UU(\bv)$, whose rank truncation yields near-optimal approximations. We denote by $\Psvd{\mathbf{v}}{\rr{r}} \mathbf{v}$ the result of truncating a $\Hcal$SVD of $\bv$ to ranks $\rr{r}$.
Using computable upper bounds $\lambda_{\rr{r}}(\bv)$ for $\|\bv - \Psvd{\mathbf{v}}{\rr{r}}\bv\|$, one can
determine near-minimal ranks
$
\rsvd(\bu,\eta) \in \argmin\bigl\{|\rr{r}|_\infty: \rr{r}\in \Rank,\;  {  \lambda_{\rsvd(\mathbf{u}, \eta)}(\bu) } \leq \eta\bigr\}
$
that ensure the validity of  a given accuracy tolerance $\eta>0$, which we use to define the recompression operator
$\hatPsvd{\eta} \bv:= \Psvd{\bv}{\rsvd(\bv,\eta)}\bv$.

The definition of a {\em coarsening} operator producing near-minimal supports of mode frames in a sense to be made precise
later, is a little more involved and based on the notion of {\em tensor contractions} which, for $i=1,\ldots,d$, are given by
$$
 \pi^{(i)}(\mathbf{v}) = \bigl( \pi^{(i)}_{\nu_i} (\mathbf{v})
 \bigr)_{\nu_i\in\nabla}  
 :=\biggl(  \Bigl(\sum_{%
 \nu_1,\ldots,\nu_{i-1},\nu_{i+1},\ldots,\nu_d} \abs{v_{ \nu_1,\ldots,\nu_{i-1},\nu_i, \nu_{i+1},\ldots,\nu_d}}^2 \Bigr)^{\frac{1}{2}}  \biggr)_{\nu_i \in\nabla} \in \ell_2(\nabla).
$$
A naive evaluation of these quantities requires a $(d-1)$-dimensional summation, which would be inacceptable. However,
the identity 
$$
 \pi^{(i)}_\nu (\mathbf{v}) = \Bigl( \sum_{k}
 \bigabs{\mathbf{U}^{(i)}_{\nu, k}}^2 \bigabs{\sigma^{(i)}_{k}}^2
 \Bigr)^{\frac{1}{2}} ,
$$
where $\sigma^{(i)}_k$ are the mode-$i$ singular values and $\{ \bU^{(i)}_k\}$ the corresponding mode frames from $\UU(\bv)$, facilitates an evaluation at a cost proportional to $\rank_i(\bv)$ for each $\nu \in \spl{2}(\nabla)$, see
\cite{BD}. The quantities
$$
 \supp_i(\bv) := \supp\big(\pi^{(i)}(\bv)\big) 
$$
allow one to quantify the the actual number of nonzero entries of mode frames, and we have $\supp \bv \subseteq \bigtimes_{i=1}^d \supp_i(\bv)$.
With the aid of a total ordering of the entries of all $\pi^{(i)}(\bv)$, $i=1,\ldots,d,$ one can find for a given $\bv$ a product set 
$\Lambda(\bv;N)$, with sum of coordinatewise cardinalities at most $N$, such that the restriction $\Restr{\Lambda(\mathbf{v};N)}$ of $\bv$ to $\Lambda(\bv;N)$
(meaning that the entries $v_\nu$ are set to zero for $\nu\notin \Lambda(\bv;N)$) satisfies
$$
 \norm{\mathbf{v} - \Restr{\Lambda(\mathbf{v};N)}\bv} \leq \mu_N(\bv) \leq 
 \sqrt{d}\, \inf \bigl\{  \norm{\mathbf{v} -
 \Restr{\hat\Lambda}\mathbf{v}} \colon {\hat\Lambda = \hat\Lambda_1\times\cdots\times \hat\Lambda_d, \, \textstyle\sum_i\#(\hat\Lambda_i)\le N}  \bigr\} \,,
 $$
where the error estimate $\mu_N(\bv)$ can be computed directly from the sequences $\pi^{(i)}(\bv)$. Setting %
 $N(\bv,\eta) := \min\bigl\{ N\colon \mu_N(\bv) \leq \eta
  \bigr\}$, %
we define   the {\em thresholding} procedure
\begin{equation}
\label{eq:tensorcoarsen}
\hatCctr{\eta} (\mathbf{v}) :=  \Restr{\Lambda(\mathbf{v};N(\bv;\eta))} \mathbf{v}\,.
\end{equation}

To assess the performance of the recompression and coarsening operators $\hatPsvd{\eta}, \hatCctr{\eta}$, as in
\cite{BD} we define
\[
\sigma_{r,\Hcal}(\bv):= \inf\,\bigl\{\norm{ \bv- \bw} \,:\; \bw \in
\Hcal(\rr{r}) \text{ with $\rr{r}\in\Rank$, $\abs{\rr{r}}_\infty \leq r$} \}\,,
\]
and, for a given {\em growth sequence} 
$\ga = \bigl(\ga(n)\bigr)_{n\in \N_0}$ with $\ga(0)=1$
and $\ga(n)\to\infty$ as $n\to\infty$, we consider
$$
\Acal(\ga)= \AH{\ga}:=  { \bigl\{\bv\in {{\spl{2}(\nabla^d)}} : \sup_{r\in\N_0} %
\ga({r})\,
\sigma_{r,\Hcal}(\bv)=:\abs{\bv}_{\AH{\ga}} { <\infty }\bigr\} },
$$ 
where we set
$\norm{\bv}_{\AH{\ga}}:= \norm{\bv} + \abs{\bv}_{\AH{\ga}}$. We always require that
$
\garatio:= \sup_{n\in\N} \ga(n)/\ga(n-1)<\infty \,,
$
which covers at most exponential growth. 
Thus, hierarchical ranks of size at most $\gamma^{-1}(\abs{\bv}_{\AH{\ga}}/\eta)$ suffice to 
approximate $\bv \in \Acal(\ga)$ within accuracy $\eta$.

Similarly, defining the error of best $N$-term approximation
$$
  \sigma_N(\mathbf{v}) := \inf_{\substack{\Lambda\subset\nabla^{\hat d}\\
  \#\Lambda\leq N}} \norm{\mathbf{v} - \Restr{\Lambda} \mathbf{v}} ,
$$
we consider for $s>0$ the classical {\em approximation classes} $\As = \As(\nabla^{\hat d})$, $\hat d\in\N$, comprised of all $\bv\in \ell_2(\nabla^{\hat d})$
for which the quasi-norm
\begin{equation*}
  \norm{\mathbf{v}}_{\As(\nabla^{\hat d})} := \sup_{N\in\N_0} (N+1)^s
  \sigma_N(\mathbf{v})  
\end{equation*}
is finite.  Hence, using this concept for $\hat d=1$, when the mode frames belong to $\As(\nabla)$,
they can be approximated within accuracy $\eta$ by finitely supported vectors of size ${\cal O}(\eta^{-1/s})$.

 The relevant facts describing the performance of $\hatPsvd{\eta}$ and $\hatCctr{\eta}$ can be summarized as follows \cite{BD}.
 
 \begin{thrm}
\label{lmm:combined_coarsening}
Let $\mathbf{u}, \mathbf{v} \in \spl{2}(\nabla^d)$ with
$\mathbf{u}\in\AH{\ga}$, $\pi^{(i)}(\mathbf{u}) \in \cA^s$ for
$i=1,\ldots,d$, and $\norm{\mathbf{u}-\mathbf{v}} \leq \eta$. Let 
$\constsvd = \sqrt{2d-3}$ and $\constcrs = \sqrt{d}$. Then, for any fixed $\alpha >0$,
\begin{equation*}
\mathbf{w}_{\eta} := \hatCctr{\constcrs
  (\constsvd+1)(1+\alpha)\eta} \bigl(\hatPsvd{\constsvd
(1+\alpha)\eta} (\mathbf{v}) \bigr) 
\end{equation*}
satisfies
\begin{equation}
\label{eq:combinedcoarsen_errest}
  \norm{\bu - \bw_\eta} \leq C(\alpha,\constsvd,\constcrs)\, \eta \,,
\end{equation}
where $ C(\alpha,\constsvd,\constcrs) :=  \bigl( 1 + \constsvd(1+\alpha) + \constcrs (\constsvd+1)(1+\alpha)\bigr)$, as well as
\begin{equation}\label{eq:combinedcoarsen_rankest}
 \abs{\rank(\bw_\eta)}_\infty \leq \ga^{-1}\bigl(\garatio
 \norm{\bu}_{\AH{\ga}}/(\alpha \eta)\bigr)\,,\qquad \norm{\bw_\eta}_{\AH{\ga}}
 \leq C_1 \norm{\bu}_{\AH{\ga}} ,
\end{equation}
with $C_1 = (\alpha^{-1}(1+\constsvd(1+\alpha)) + 1)$ and
\begin{equation}\label{eq:combinedcoarsen_suppest}
\begin{aligned}
 \sum_{i=1}^d \#\supp_i (\mathbf{w}_\eta) &\leq {2 \eta^{-\frac{1}{s}} d\, \alpha^{-\frac{1}{s}} } \Bigl(
 \sum_{i=1}^d \norm{ \pi^{(i)}(\bu)}_{\As} \Bigr)^{\frac{1}{s}}  \,, \\
   \sum_{i=1}^d \norm{\pi^{(i)}(\bw_\eta)}_{\As} &\leq C_2 \sum_{i=1}^d
   \norm{\pi^{(i)}(\bu)}_{\As} ,
 \end{aligned}
\end{equation}
with $C_2 = 2^s(1+ {3^s}) + 2^{{4s}} { \alpha^{-1} \bigl( 1+  \constsvd(1+\alpha) + \constcrs  (\constsvd + 1)(1+\alpha) \bigr)  }  d^{\max\{1,s\}}$.
\end{thrm}

\begin{rmrk}\label{rem:CRcomplexity}
Both $\hatPsvd{\eta}$ and $\hatCctr{\eta}$ require a hierarchical singular value decomposition of their inputs.
For a finitely supported $\bv$ given in hierarchical format, the number of operations required for obtaining such a decomposition is bounded, up to a fixed multiplicative constant, by $d\abs{\rank(\bv)}_\infty^4 + \abs{\rank(\bv)}_\infty^2 \sum_{i=1}^d \# \supp_i\bv$, see also \cite{Grasedyck:10}.
\end{rmrk}

\section{Asymmetric Preconditioning}\label{sec:asymprecond}
\subsection{Transformation to Well-Conditioned Systems}\label{ssec:ddep}
Following \cite{BD,BD2},   to iteratively solve \eqref{Tsystem} and hence \eqref{1.1}, we first need 
to  precondition the operator $\bT$ to obtain a well-conditioned operator equation on $\spl{2}(\nabla^d)$. A natural way of doing this is to exploit the mapping properties \eqref{mappinH1}
in combination with the fact that a suitable diagonal scaling of the $\spL{2}(\Omega)$-wavelet basis gives rise to a
{\em Riesz basis} for $\spH{1}_0(\Omega)$. To describe this
we choose for $i=1,\ldots,d$,  the scaling weights $\omi{i}{\nu_i}$, $\nu_i\in\nabla$, such that 
\begin{equation}\label{omiscale}
\omi{i}{\nu_i} \sim 2^{\abs{\nu_i}}
\end{equation}
with uniform constants, and set
\begin{equation}
\label{can-scaling}
\om{\nu} := \om{\nu_1,\ldots,\nu_d} = \Big(\sum_{i=1}^d (\omi{i}{\nu_i})^2\Big)^{1/2},\quad \nu\in\nabla^d.
\end{equation}
With this sequence, we define the diagonal scaling operator
\begin{equation}
\label{bS}
\Sc = \big(\om{\nu}\delta_{\nu,\mu}\big)_{\nu,\mu\in\nabla^d} \,.
\end{equation}
In addition, for later reference, we define  for $\tau\in\R$ and $i=1,\ldots,d$ on the one hand the coordinatewise scaling operators 
$\Sc^{\tau}_{i}:  \R^{\nabla^d} \to \R^{\nabla^d}$ by
\begin{equation}
\label{Si}
\Sc^{\tau}_{i} \bv := \bigl( \omi{i}{\nu_i}^\tau v_\nu \bigr)_{\nu\in \nabla^d}
\quad  \text{and}\quad \Sc_{i} := \Sc_{i}^1\,,
\end{equation}
and on the other hand, the corresponding \emph{low-dimensional} scaling operators $\Sci{i}^\tau \colon \R^\nabla \to \R^\nabla$ by
\begin{equation}\label{eq:tensor_rescaling}   
\Sci{i}^\tau \mathbf{\hat v} := \bigl( \omi{i}{\nu_i}^\tau \hat v_{\nu_i} \bigr)_{\nu_i\in\nabla} \quad  \text{and}\quad \Sci{i} := \Sci{i}^1  \,.
\end{equation}

It is well-known that under the above assumptions on the basis $\{\Psi_\nu\}$, the rescaled mapping $\bS^{-1}\bT\bS^{-1}$ is an
isomorphism from $\ell_2(\nabla^d)$ onto itself, which is related to the fact that $u\in \spH{1}_0(\Omega)$ if and only if $\bS\bu\in \ell_2(\nabla^d)$,
where $\bu$ is the wavelet coefficient tensor with respect to the $\spL{2}(\Omega)$-basis. Note that this implies, in particular,
that for each $\nu\in \nabla^d$ the quantity $(\bT\bu)_\nu$ is well-defined when the corresponding function $u$ belongs to
$\spH{1}_0(\Omega)$. These facts have been exploited in \cite{BD2} by replacing \eqref{Tsystem} by the 
(symmetrically) preconditioned system $\bS^{-1}\bT\bS^{-1} \bu^{(1)} = \bS^{-1}\bg$, i.e., one actually solves for
the $\spH{1}(\Omega)$-scaled coefficient array $\bu^{(1)} = \bS\bu$.

In this paper we follow a different direction, seeking directly the $\spL{2}(\Omega)$-wavelet coefficients $\langle u, \Psi_\nu\rangle$ of the solution $u$ to \eqref{1.1},
and thus of \eqref{Tsystem}. Here we exploit that $(\bT\bu)_\nu$ is still well-defined for arbitrary $\bu\in\spl{2}(\nabla^d)$ provided that the wavelet basis functions are sufficiently regular. Our approach is based on the following facts.
\begin{thrm}
\label{thm:S2T}
Assume that the univariate wavelet basis $\{ \psi_\nu\}$ is $\spL{2}(0,1)$-orthonormal and that $\{ 2^{-2\abs{\nu}} \psi_\nu \}$ is a Riesz basis of 
$\spH{2}(0,1)\cap \spH{1}_0(0,1)$.
Then, for $\bS,\bT$ defined by \eqref{bS}, \eqref{eq:unrescaled_tuckersum}, respectively, the infinite matrix $\bS^{-2}\bT$
is an isomorpishm from $\ell_2(\nabla^d)$ onto itself, i.e., there exist constants $0<c \leq C<\infty$ such that
\begin{equation}
\label{norm-eq}
 c \norm{\bv} \leq \norm{\Sc^{-2}\bT \bv}  \leq C \norm{\bv} \,, \quad \bv \in{\spl{2}(\nabla^d)} .
\end{equation}
Moreover, when $M$ is diagonal and $\omi{i}{\nu_i}\sim \sqrt{m_{ii}}\, 2^{\abs{\nu_i}}$, the constants $c, C$ are independent of the spatial dimension $d$.
\end{thrm}

\begin{rmrk}\label{remark:H2rbcond}
The property that the univariate rescaled wavelet basis is a  Riesz basis for $\spH{2}(0,1)\cap \spH{1}_0(0,1)$,  required in Theorem \ref{thm:S2T}, is satisfied, in particular, if the univariate wavelet (or multiwavelet) basis functions are $\spL{2}$-orthonormal, piecewise polynomial, belong to $\mathrm{C}^1(0,1)$, and vanish at the endpoints of the  interval, provided that  the scaling functions have the following  polynomial reproduction property: for each level and any closed subinterval of $[0,1]$ on which the scaling functions of that level are polynomial,
for such subintervals contained in the interior, all polynomials of degree two are reproduced, while on those such subintervals containing an endpoint only those polynomials
are reproduced that vanish at that endpoint.
Note that a piecewise polynomial in $\mathrm{C}^1(0,1)$ belongs to $\spH{2+\tau}(0,1)$ for any $\tau < \frac12$. With the above properties the validity of suitable inverse and direct estimates can be verified which, combined with orthogonality, imply the required Riesz basis property, see \cite{Dahmen:96}.
\end{rmrk}
\begin{proof}
Note that
$$
V:=   \spH{1}_0(\Omega) \cap \spH{2}(\Omega) = \bigcap_{i=1}^d \spL{2}(0,1)\otimes\cdots  \otimes (\spH{1}_0(0,1) \cap \spH{2}(0,1) )\otimes \cdots \otimes \spL{2}(0,1) .
$$
Thus the rescaled wavelets $\omega_\nu^{-2}\Psi_\nu$, $\nu\in\nabla^d$,
form a Riesz basis for $V$, and they are therefore the dual of a Riesz-basis for $V' =(\spH{1}_0(\Omega) \cap \spH{2}(\Omega))'$, i.e., for $w\in V'$ one has $\|w\|_{V'}\sim \|(\langle w, \omega_\nu^{-2}\Psi_\nu\rangle)_{\nu\in\nabla^d}\|_{\spl{2}(\nabla^d)}$. Due to the convexity of $\Omega$
and the fact that $M$ is constant, the operator $A$ maps $V$ one-to-one and onto  $\spL{2}(\Omega)$ and hence, by duality,
from $\spL{2}(\Omega)$ onto $V'$. Since $(\langle Au, \omega_\nu^{-2}\Psi_\nu\rangle)_{\nu\in\nabla^d}= \bS^{-2}\bT\bu$ and
$\|Au\|_{V'} \sim \|u\|_{\spL{2}(\Omega)}\sim \|\bu\|_{\spl{2}(\nabla^d)}$,  the norm equivalence \eqref{norm-eq} follows. 

 To prove the rest of the assertion, by the choice of $\omi{i}{\nu_i}$ (cf.\ \cite{Dijkema:09,BD2}), it suffices to confine the discussion to
the Laplacian on $\spH{1}_0(\Omega)$, where 
$$   
\bT = \sum_{i=1}^d \bT^{(i)}_{2} := \sum_{i=1}^d \id_1 \otimes \cdots\otimes\id_{i-1} \otimes\bT_2 \otimes \id_{i+1} \otimes\cdots\otimes\id_{d} 
$$
with $\bT_2$ as in \eqref{T1}. In order to estimate the constants $c, C$ in \eqref{norm-eq} in this case,
we hence need to find bounds for the extreme singular values of $\Sc^{-2}\bT$, or equivalently, the eigenvalues of $\Sc^{-2}\bT \bT^* \Sc^{-2} = \Sc^{-2}\bT^2 \Sc^{-2}$.
To this end, recall that
$$ 
 \Sc^2 = \sum_{i=1}^d \Sc_i^2 = \sum_{i=1}^d \id_1 \otimes \cdots\otimes\id_{i-1} \otimes\Sci{i}^2 \otimes \id_{i+1} \otimes\cdots\otimes\id_{d} \,.
$$
The desired statement follows if we can show that, for any \emph{compactly supported} $\bv$,
\begin{equation}
\label{eq:sv_est}   
c^2 \langle \Sc^4 \bv,\bv\rangle \leq \langle \bT^2\bv,\bv \rangle
  \leq C^2  \langle \Sc^4 \bv,\bv\rangle  
\end{equation}
with suitable $c,C$, since then the singular values of $\Sc^{-2} \bT$ are contained in $[c,C]$.

We now estimate the summands in the expansions
$$    
\Sc^4 =  \sum_{i,j=1}^d \Sc_i^2  \Sc_j^2 \,,\quad \bT^2 = \sum_{i,j=1}^d \bT^{(i)}_{2} \bT^{(j)}_{2}  
$$
separately and then add the different contributions to obtain \eqref{eq:sv_est} with $c,C$ independent of $d$.
If $i\neq j$, we have $\tilde c, \tilde C$ such that
$$  
\tilde c^2 \,  \Sci{i}^2\otimes \Sci{j}^2 \leq \bT_{2} \otimes \bT_{2}  \leq  \tilde C^2 \, \Sci{i}^2\otimes \Sci{j}^2 
$$
in the sense, analogously to \eqref{eq:sv_est}, of inner products with compactly supported sequences on $\nabla^2$; here we need only that $\{ \psi_\nu\}$ is an orthonormal basis of $\spL{2}(0,1)$ and $\{ 2^{-\abs{\nu}} \psi_\nu\}$ is a Riesz basis of $\spH{1}_0(0,1)$.

The case $i=j$ is, however, more involved: in general, we do \emph{not} have $\tilde c^2 \,  \Sci{i}^4 \leq \bT_{2}^2  \leq  \tilde C^2 \, \Sci{i}^4$ with the same $\tilde c, \tilde C$. Now we use in addition that  $\{ 2^{-2\abs{\nu}} \psi_\nu\}$ is a Riesz basis of $V_1 :=\spH{2}(0,1) \cap \spH{1}_0(0,1)$.
Using also $\spL{2}$-orthonormality, we obtain
$$ 
 \bT_{2,\mu\nu}^2 = \sum_\lambda \langle \psi_\mu',\psi_\lambda'\rangle\langle\psi_\lambda',\psi_\nu'\rangle
  = \sum_\lambda \langle \psi_\mu'',\psi_\lambda\rangle\langle\psi_\lambda,\psi_\nu''\rangle
  = \langle \psi_\mu'',\psi_\lambda''\rangle \,.  
$$
We now verify that $\norm{u''}_{\spL{2}(0,1)}$ is a norm on $V_1$ by comparison with  the standard norm $\norm{u}_{V_1}^2 := \norm{u}_{\spL{2}(0,1)}^2 + \norm{u'}_{\spL{2}(0,1)}^2 + \norm{u''}_{\spL{2}(0,1)}^2$. By the Poincar\'e inequality, $\norm{u''}_{\spL{2}(0,1)} \gtrsim \norm{u' - \int_0^1 u'\,dx}_{\spL{2}(0,1)} = \norm{u'}_{\spL{2}(0,1)}$, where we have used that $\int_0^1 u'\,dx = 0$ as a consequence of $u\in \spH{1}_0(0,1)$. By the Poincar\'e-Friedrichs inequality, $\norm{u'}_{\spL{2}(0,1)} \gtrsim \norm{u}_{\spL{2}(0,1)}$. Hence $\norm{u''}_{\spL{2}(0,1)} \sim \norm{u}_{V_1}$. By the Riesz basis property for $V_1$, we thus have $\hat c, \hat C$ such that
$$
\hat c^2 \,  \Sci{i}^4 \leq \bT_{2}^2  \leq  \hat C^2 \, \Sci{i}^4 \,.
$$
We thus obtain \eqref{eq:sv_est} with $c = \min\{\tilde c, \hat c\}$, $C=\max\{\tilde c, \hat C\}$, which are in particular independent of $d$.
\end{proof}

\begin{rmrk}
Clearly, unlike the symmetrically preconditioned version $\bS^{-1}\bT\bS^{-1}$ considered in \cite{BD2}, $\bS^{-2}\bT$ is in general nonsymmetric. It is generally also nonnormal, since normality would require $\norm{\bS^{-2}\bT \bv} = \norm{\bT \bS^{-2} \bv}$ for any $\bv\in\spl{2}(\nabla^d)$ 
and hence, in particular,
$$
   \sum_\mu \bigl( \omega_\mu^{-2} T_{\mu\nu} \bigr)^2 = \omega_{\nu}^{-4} \sum_\mu \bigl( T_{\mu\nu} \bigr)^2\,,\quad \nu\in\nabla^d\,,
$$
which holds only for very specific choices of $\{ \Psi_\nu\}$ (e.g., for a basis of eigenfunctions of $A$).
\end{rmrk}

Based on Theorem \ref{thm:S2T}, our envisaged numerical scheme may be viewed as a perturbed version of
the Jacobi-type iteration %
\begin{equation}\label{eq:idealized_iter}   
\bu_{j+1} = \bu_j - \omega \Sc^{-2} (\bT \bu_j - \bg) \,.  
\end{equation}

The perturbations result from approximating all quantities by finitely supported sequences and from additional low-rank approximations in hierarchical tensor format.
While the asymmetric preconditioning by $\bS^{-2}$ causes the loss of symmetry it has the following
advantage: the application of the finite rank operator $\bT$ to a finite rank iterate $\bu_j$ increases the output rank by only a little.
The scaling operator $\bS^{-2}$, however, has {\em infinite rank} so that the construction of a finite rank approximation to
the  scaled  residual $\Sc^{-2} (\bT \bu_j - \bg)$
must involve a substantial rank reduction. For finding a good compromise between accuracy and rank size, Theorem \ref{lmm:combined_coarsening}
is pivotal. Note that in the symmetric case $\bS^{-1}\bT\bS^{-1}$, the rank-inflating scaling operation has to be done twice,
with corresponding consequences concerning computational complexity. The effect of a one-sided scaling  will later be quantified,
in addition to an analytical assessment, by our numerical experiments.

Our strategy for producing an approximate finite rank residual   is similar in spirit to the approach in \cite{BD2} for
the symmetric case, namely to approximate the scaling operator $\bS^{-2}$ by a finite-rank operator.
The foundation of this approximation is given in the next section.
\subsection{Low-Rank Preconditioner}\label{ssec:precond}

Rather than approximately  applying $\bS^{-1}$ twice we find a direct finite rank approximation for $\bS^{-2}$ with the aid of the following
{\em relative error estimate for exponential sum approximation}.
\newcommand{\bd}{b}
\newcommand{\omeg}{w}

\begin{thrm}\label{thm:expsum_relerr}
Let $\delta \in (0,1)$ and
\begin{equation}
\label{eq:expapprox_defs} 0<h \leq \sup_{\bd\in (0, \pi/2)} \frac{2\pi \bd}{4 + \abs{\ln \delta} + \abs{\ln \cos \bd}} ,\quad 
\alpha(x) := \ln(1+e^x), \quad
\omeg(x) :=  (1 + e^{-x})^{-1} \,. 
\end{equation}
Let $n^+ := \ceil{ h^{-1} \abs{\ln (\frac12\delta)} }$ and
\begin{equation}
\label{eq:expapprox_sumdefs}   
   \varphi_{h,n}(t) :=  \sum_{k=-n}^{n^+} h\, w(kh)\, e^{-\alpha(kh)\, t}  \,, \quad \varphi_{h,\infty}(t) := \lim_{n\to\infty} \varphi_{h,n}(t)\,.
\end{equation}
Then
\begin{equation}\label{eq:expapprox_sinc_est}
   \bigabs{ t^{-1} - 
   \varphi_{h,\infty}(t) } 
  \leq  \delta\,t^{-1} \quad \text{for all $t\in[1,\infty)$,} 
\end{equation}
and furthermore, for any $\varepsilon > 0$ and $n \geq \ceil{ h^{-1} \abs{\ln\varepsilon} }$, we have
\begin{equation}\label{eq:ptail}
\bigabs{
  \varphi_{h,n}(t) - \varphi_{h,\infty }(t) }  \leq \varepsilon \quad \text{for all $t\in[1,\infty)$.} 
\end{equation}
Consequently, for $\eta > 0$, $T>1$, and $n \geq \ceil{ h^{-1} ( \abs{\ln\eta} + \ln T) }$, we have
\begin{equation}
\bigabs{
   \varphi_{h,n}(t) - \varphi_{h,\infty }(t) }  \leq \eta\, t^{-1}\quad \text{for all $t\in[1,T]$.} 
\end{equation}
\end{thrm}

Note that the supremum in \eqref{eq:expapprox_defs} is attained for any $\delta > 0$.

\begin{proof}
Our starting point is the integral representation (cf.\ \cite{Hackbusch:06-1})
$$  \frac{1}{r} = \int_0^\infty e^{-rt}\,dt = \int_{-\infty}^\infty e^{-r\ln(1+e^x)} \frac{dx}{1+e^{-x}} \,. $$
The integrand is analytic in the strip $\{ x+iy \colon x\in\R, \abs{y}< \pi/2\}$.
Our aim is to apply \cite[Theorem 3.2.1]{Stenger:93}, which gives
$$
\biggabs{ \frac{1}{t} - 
   \sum_{k\in\Z} h\, \omega(kh) e^{-\alpha(kh)\, t} } \leq
   N_d \frac{e^{-\pi \bd/h}}{2 \sinh(\pi \bd/h)} \,,
$$
where
$$
  N_\bd :=  \int_\R \biggabs{\frac{e^{-t \ln(1+e^{x+i\bd})}}{1+e^{-(x+i\bd)}} }  \,dx
    +  \int_\R \biggabs{ \frac{e^{-t \ln(1+e^{x-i\bd})}}{1+e^{-(x-i\bd)}} } \,dx
    \,,\quad \bd \in(0,\pi/2)\,.
$$
We thus need a suitable estimate for $N_\bd$. Note that $\abs{1 + e^{x\pm i\bd}}^2 \geq 1 + e^{2x} \geq \frac12 (1 + e^x)^2$ and $\abs{\ln(1+e^{x\pm i\bd})}=\ln(1+e^x \cos \bd)$. Furthermore, for $x\geq 0$ we obtain $1+e^x \cos \bd \geq e^{x \cos \bd}$ from comparing the respective series expansions, hence $\ln(1+e^x \cos \bd) \geq x \cos \bd$ for $x\geq 0$. For $x\leq 0$, we observe that $\ln(1+y)\geq \frac12 y$ for any $y\in[0,1]$, and hence $\ln(1+e^x\cos \bd)\geq \frac12 x \cos \bd$ for $x\leq 0$.

For such $\bd$, we now obtain
$$  \int_{\R^+} \biggabs{ \frac{e^{-t \ln(1+e^{x\pm i\bd})}}{1+e^{-(x\pm i\bd)}} } \,dx
   \leq   2 \int_{\R^+} \frac{ e^{-t x\cos \bd} }{1 + e^{-x}} \,dx
    \leq  2 \int_{\R^+}  e^{-t x\cos \bd} \,dx
   \leq 2 (t \cos \bd)^{-1} $$
as well as
\begin{equation*}
    \int_{\R^-} \biggabs{ \frac{e^{-t \ln(1+e^{x\pm i\bd})}}{1+e^{-(x\pm i\bd)}} } \,dx 
    \leq  2 \int_{\R^+} \frac{e^{-\frac{t}{2} e^{-x} \cos \bd}}{1+e^{x}}\,dx \\
  = 2 \int_0^1 \frac{e^{-\frac{t}{2} \xi \cos \bd}}{(1+\xi^{-1})\xi} \,d\xi 
  \leq 4(t \cos \bd)^{-1}    \,,
\end{equation*}
where we have used the substitution $x=-\ln \xi$.

Applying \cite[Theorem 3.2.1]{Stenger:93}, we thus obtain
$$
\biggabs{ \frac{1}{t} - 
   \sum_{k\in\Z} h\, \omeg(kh) e^{-\alpha(kh)\, t} } \leq 12 (t \cos \bd)^{-1} \frac{e^{-\pi \bd/h}}{2 \sinh(\pi \bd/h)} \leq 24 (t \cos \bd)^{-1}   e^{-2 \pi \bd/h} \leq \frac12 t^{-1} \delta
$$
for the range of $h$ given in the assertion. Here we have used that in particular, $h \leq 2\pi \bd/ \ln 2$, which gives ${e^{-\pi \bd/h}}/{(2 \sinh(\pi \bd/h))} \leq 2 e^{-2\pi \bd/h}$, and that $\ln 48 < 4$.

The estimates for $n^+$ and $n$ follow from the decay of the integrand on $\R$: on the one hand, we have 
$$ 
 \sum_{k> n^+} h\, \omeg(kh)  e^{-\alpha(kh)\, t} \leq  h \int_{n^+}^\infty  e^{-txh}\, dx \leq t^{-1} \int_{n^+ ht}^\infty e^{-x}\,dx \leq t^{-1} e^{-n^+ h} \,. 
 $$
The expression on the right hand side is bounded by $\frac12 t^{-1}\delta$ for $n^+ \geq h^{-1} (\ln 2 + \abs{\ln \delta})$, which yields \eqref{eq:expapprox_sinc_est}.
On the other hand, 
$$ 
  \sum_{k < -n} h\, \omeg(kh)  e^{-\alpha(kh)\, t} 
    \leq   \int_{n h}^\infty e^{-x} \,dx  
      \leq e^{-n h}\,,
$$
and the expression on the right hand side is bounded by $t^{-1}\eta$ for all $t\in[1,T]$ for 
$n \geq h^{-1}(\abs{\ln\eta} + \ln T)$.
\end{proof}

\begin{rmrk}
A related but slightly different relative error bound, for approximation of $t^{-1}$ on $(0,1]$, was derived for a different purpose in \cite{Beylkin-Monzon}. The main difference is that the above bound allows us to realize arbitrarily good approximations to a scaling operator equivalent to
$\bS^{-2}$ by simply adding additional separable terms while keeping the upper summation index $n^+$ fixed. This is a significant advantage regarding implementation.

In other works, preconditioners for low-rank tensor methods for fixed discretizations of second-order problems have been proposed, for instance \cite{Andreev:12,Khoromskij:09,Kressner:11,Ballani:13}. However, these have not been analyzed in their overall effect on the complexity of the solution process.
\end{rmrk} 

In what follows, we {\em fix}  $\delta\in(0,1)$ and $h$, $n^+$ as in Theorem \ref{thm:expsum_relerr}. For  the corresponding $\varphi_{h,n}$ and $\varphi_{h,\infty}$ we define
\begin{equation*}
  p_{n,\nu} :=  \omin^{-2}\, \varphi_{h,n}\bigl( (\omega_\nu / \omin)^2 \bigr) 
  \,,\quad
  p_{\nu} := \lim_{n\to\infty} p_{n,\nu} = \omin^{-2}\, \varphi_{h,\infty}\bigl( (\omega_\nu / \omin)^2 \bigr) \,,
\end{equation*}
where $\omin := \min_{\nu\in\nabla^d} \omega_\nu$. We then  set
\begin{equation}
\label{Pn}
  \Pc := \diag(p_\nu)\,,\quad  \Pc_n := \diag(p_{n,\nu})\,. 
\end{equation}
Theorem \ref{thm:expsum_relerr} states that $\Pc$, $\Pc_n$ have the properties
$$   
\norm{ (\Pc - \Sc^{-2} ) \Sc^2} \leq \delta \,, \qquad  
      \norm{ (\Pc - \Pc_n ) \Sc^2 \Restr{\Lambda_T} } \leq \eta
       \quad\text{for $n\geq \ceil{ h^{-1} ( \abs{\ln\eta} + \ln T) }$.}
$$
In other words, $\Pc$ is an approximation of $\Sc^{-2}$ with a \emph{relative} error bound $\delta$, and $\Pc_n$ provides a finite-rank approximation to $\Pc$ for any prescribed relative error bound $\eta$ on compactly supported sequences.
We shall use $\Pc$ which, in turn,  is   approximated by $\Pc_n$, as a substitute for $\Sc^{-2}$  in \eqref{eq:idealized_iter} when solving
$$  
\bT \bu = \bg  
$$
by a Jacobi-type iteration. The modified idealized iteration thus has the form
\begin{equation}
\label{jacobi}   
\bu_{j+1} = \bu_j - \omega \Pc (\bT \bu_j - \bg) \,.  
\end{equation}
Setting $\bA := \Pc\bT$, $\bbf := \Pc\bg$, this iteration will be realized in the perturbed form 
$$  
 \bu_{j+1} = \bu_j - \omega \mathbf{r}_j \,,\quad \mathbf{r}_j \approx (\Pc\bT) \bu_j - \Pc\bg 
$$
with  a suitable approximation $\mathbf{r}_j$, involving $\Pc_n$, of the scaled residual.

\section{Analysis of an Adaptive Method with Error Control in $\spL{2}$}\label{sec:analysis}

\subsection{The Adaptive Scheme}\label{ssec:3.1}
The adaptive scheme to be proposed next has the following routines as main constituents:
{\begin{itemize}
\item
$\recompress(\mathbf{v};
\eta)$, realizing the projection $\hatPsvd{\eta} (\bv) := \Psvd{\bv}{\rsvd(\bv,\eta)}\bv$ from Section \ref{sssec:1.2.2} with target accuracy $\eta$;
\item 
$\coarsen(\bv;\eta)$, realizing the coarsening operator $\hatCctr{\eta} (\mathbf{v})$ from \eqref{eq:tensorcoarsen};
\item
$\rhs(\eta)$, producing an $\eta$-accurate approximation to the right hand side $\bbf$;
\item 
$\apply(\bv;\eta)$, which yields $\bw_\eta$ of finite support and ranks such that $\norm{\bA\bv - \bw_\eta}\leq \eta$.
\end{itemize}
For a discussion of the first three routines we refer to \cite{BD,BD2} and defer
the precise description of $\apply(\bv;\eta)$ to Section  \ref{ssec:opappr}. We formulate next the perturbed version of the idealized iteration
\eqref{jacobi} in Algorithm \ref{alg:tensor_opeq_solve}:}

\begin{algorithm}[!ht]
\caption{$\quad \mathbf{u}_\varepsilon = \solve(\mathbf{A},
\mathbf{f}; \varepsilon)$} \begin{algorithmic}[1]
\Require $\Bigg\{$\begin{minipage}{12cm}$\omega >0$ and $\rho\in(0,1)$ such that
$\norm{\id - \omega\Pc^{\frac12}\bT\Pc^{\frac12}} \leq \rho$,\\
$c_\bA \geq \norm{\bA^{-1}}$, $\varepsilon_0 \geq c_\bA \norm{\mathbf{f}}$, \\
$ \kappa_1, \kappa_2, \kappa_3 \in (0,1)$ with $\kappa_1 +
\kappa_2 + \kappa_3 \leq 1$, and $\beta_1 \geq 0$, $\beta_2 > 0$.\end{minipage}
\Ensure $\mathbf{u}_\varepsilon$ satisfying $\norm{\mathbf{u}_\varepsilon -
\mathbf{u}}\leq \varepsilon$.
\State $\mathbf{u}_0 := 0$, $k:= 0$
\While{$2^{-k} \varepsilon_0 > \varepsilon$}
\State $\eta_{k,0} := \rho  2^{-k} \varepsilon_0$
\State $\mathbf{w}_{k,0}:=\mathbf{u}_k$
\State $\mathbf{r}_{k,0} := \apply( \mathbf{w}_{k,0} ; \frac{1}{2}\eta_{k,0})
- \rhs(\frac{1}{2}\eta_{k,0})$ 
\State $j \gets 0$
\While{$c_\bA (\norm{\mathbf{r}_{k,j}} + \eta_{k,j}) > \kappa_1 2^{-(k+1)} \varepsilon_0$} \label{alg:looptermination_line}
\State $\mathbf{w}_{k,j+1} := \coarsen\bigl(\recompress(\mathbf{w}_{k,j} - \omega \mathbf{r}_{k,j} ;
\beta_1 \eta_{k,j}); \beta_2 \eta_{k,j} \bigr)$ \label{alg:tensor_solve_innerrecomp}
\State $j\gets j+1$.
\State $\eta_{k,j} := \rho^{j+1} 2^{-k} \varepsilon_0$
\State $\mathbf{r}_{k,j} := \apply( \mathbf{w}_{k,j} ; \frac{1}{2}\eta_{k,j})
- \rhs(\frac{1}{2}\eta_{k,j})$ 
\EndWhile
\State $\mathbf{u}_{k+1} := \coarsen\bigl(\recompress(\mathbf{w}_{k,j};
\kappa_2 2^{-(k+1)} \varepsilon_0) ; \kappa_3 2^{-(k+1)}
\varepsilon_0\bigr)$\label{alg:cddtwo_coarsen_line} 
\State $k \gets k+1$
\EndWhile
\State $\mathbf{u}_\varepsilon := \mathbf{u}_k$ 
\end{algorithmic}
\label{alg:tensor_opeq_solve}
\end{algorithm}

\subsection{Convergence Analysis}
\label{ssec:convan}
We address first the convergence of the idealized iteration  \eqref{jacobi}.
\begin{rmrk}
Since $\Omega$ is bounded, $A$ has a purely discrete spectrum and all eigenfunctions   of $A$ belong to  $\spH{1}_0(\Omega)$. As a  consequence, $\bA = \Pc \bT$ and $\Pc^{\frac12}\bT\Pc^{\frac12}$ have the same spectrum, where we recall that $\Pc^{\frac 12}$ is spectrally equivalent
to $\bS^{-1}$.
\end{rmrk}

Let $\omega>0$ be chosen such that $\rho:=\norm{\id - \omega\Pc^{\frac12}\bT\Pc^{\frac12}} < 1$. Since the eigenvalues of $\id - \omega\Pc^{\frac12}\bT\Pc^{\frac12}$ and $\bC:= \id - \omega\bA$ coincide, we have
\begin{equation}
  \lim_{k\to\infty} \norm{\bC^k}^{\frac1k} = \rho \,.
\end{equation}
Consequently, for an arbitrarily fixed $\tilde\rho$ with $\rho<\tilde\rho<1$, this implies the following: there exist $K\in\N$ and $B>0$ such that
\begin{equation}
\label{K}
     \norm{\bC^k} \leq \tilde\rho^k \quad\text{for $k > K = K(\tilde\rho)$,} \qquad  \norm{\bC^k} \leq B = B(\tilde \rho)\quad\text{for $k \leq K$,}
\end{equation}
which confirms the convergence of \eqref{jacobi}. It now remains to account for the additional perturbations in Algorithm \ref{alg:tensor_opeq_solve}.

\begin{prpstn}
\label{prop:converge}
For any given target accuracy $\varepsilon >0$, Algorithm \ref{alg:tensor_opeq_solve} terminates after finitely many steps and
yields a finitely supported  tensor $\bu_\varepsilon$ with finite hierarchical ranks, satisfying
\begin{equation}
\label{targetsol}
\norm{u - u_\varepsilon}_{\spL{2}(\Omega)}= \norm{\bu - \bu_\varepsilon} \le \varepsilon,
\end{equation}
where $u$ is the exact solution of \eqref{1.1}, whose $\spL{2}$-wavelet coefficient array $\bu$ satisfies \eqref{Tsystem},
and $\bu_\varepsilon$ is the coefficient tensor of $u_\varepsilon$.
\end{prpstn}
\begin{proof}
The argument is similar to that in \cite{BD} and differs only in the treatment of the inner loop
  between steps $7$ and $12$ in Algorithm \ref{alg:tensor_opeq_solve}.
  For convenience we briefly sketch the induction argument that shows that $\|\bu_k-\bu\|\le 2^{-k}\varepsilon_0$.
To that end, since by step 5,
$$
\norm{\bw_{k,j}-\bu} \le \norm{\bA^{-1}} \norm{\bA \bw_{k,j} - \bbf} \le c_\bA(\norm{\mathbf{r}_{k,j}} +\eta_{k,j}),
$$
condition \ref{alg:looptermination_line} ensures that when exiting the inner loop at step 12, the approximation $\bw_{k,j}$ satisfies $\|\bw_{k,j}-\bu \|\le \kappa_12^{-(k+1)}\varepsilon_0$. To see that $c_\bA(\|\mathbf{r}_{k,j}\| +\eta_{k,j})$ indeed becomes as small as one wishes when $j$ increases, one derives from
steps 5, 8, and the definition of $\eta_{k,j}$ in step 10, 
that the iterates $\bw_{k,j}$ satisfy a relation of the form
$$
   \bw_{k,j+1} = \bw_{k,j} - \omega \bA \bw_{k,j} + \omega \bbf + \mathbf{z}_{k,j}  \,,
$$
where $\|\mathbf{z}_{k,j} \|\le (\beta_1 + \beta_2 +\omega)\eta_{k,j} =: \varepsilon_{k,j}$.
Using $\bw_{k,j+1} -\bu = \bC (\bw_{k,j} - \bu) + \mathbf{z}_{k,j}$, we thus obtain, for $j >K$,
\begin{align*}
  \norm{\bw_{k,j} - \bu}  &\leq \norm{\bC^j} \norm{\bw_{k,0} - \bu} + \sum_{\ell = 0}^{j-1} \norm{\bC^{j-1-\ell}} \norm{\mathbf{z}_{k,\ell}}   \\
  & \leq {\tilde\rho}^{j} \norm{\bw_{k,0} - \bu}  + \sum_{\ell=0}^{j-1-K} {\tilde\rho}^{j-1-\ell} \norm{\mathbf{z}_{k,\ell}}  + B \sum_{\ell = j - K}^{j-1} \norm{\mathbf{z}_{k,\ell}} \,.
\end{align*}
Since  $\norm{\mathbf{z}_{k,\ell}} \leq \varepsilon_{k,\ell} \le (\beta_1 + \beta_2 +\omega) \rho^\ell 2^{-k}\varepsilon_0$ we
conclude that  for $\beta_3:= (\beta_1 + \beta_2 +\omega)$,
\begin{eqnarray}
\label{jlarge}
  \norm{\bw_{k,j} - \bu} &\leq &{\tilde\rho}^j \norm{\bu_k - \bu} + \big( (j-K) {\tilde\rho}^{j-1} \nonumber \\
&&\quad     + (1-\rho)^{-1} (\rho^{-K} - 1) B \rho^j\big)\beta_3   2^{-k}\varepsilon_0\nonumber \\
&\le &\big\{ {\tilde\rho}^j  + \big((j-K) {\tilde\rho}^{j-1} + (1-\rho)^{-1} (\rho^{-K} - 1) B \rho^j\big)\beta_3
\big\} 2^{-k}\varepsilon_0 \,.
\end{eqnarray}
On the other hand, observing that 
$$
\norm{\mathbf{r}_{k,j}}\le \norm{\bA\bw_{k,j} -\bbf}+\eta_{k,j} \le \norm{\bA} \norm{\bw_{k,j}-\bu}+ \eta_{k,j},
$$
we see that after at most a finite number $J$ of steps, depending only on $\bA$ (i.e.,  on the operator $A$ and the chosen wavelet basis), indeed $c_\bA(\norm{\mathbf{r}_{k,J}}+\eta_{k,J})
\le \kappa_1 2^{-(k+1)}\varepsilon_0$ holds, the inner loop terminates and hence $\norm{\bw_{k,J}-\bu} \le \kappa_1 2^{-(k+1)}\varepsilon_0$. For later reference, note that $J\leq I$ with
\begin{equation}\label{eq:Idef}
  I := \min\Bigl\{ j \geq K\colon c_\bA \Bigl( \norm{\bA} \bigl[ \tilde{\rho}^j + \bigl( (j-K) \tilde{\rho}^{j-1} + (1-\rho)^{-1} (\rho^{-K} -1) B \rho^j \bigr)\beta_3 \bigr] + 2\rho^{j+1}  \Bigr) \leq \frac{\kappa_1}2 \Bigr\} \,.
\end{equation}
Since $\kappa_1+\kappa_2+\kappa_3\le 1$, we obtain $\norm{ \bu_{k+1}-\bu } \le 2^{-(k+1)}\varepsilon$.
\end{proof}

\subsection{Operator Approximation}\label{ssec:opappr}

Our approximate application of the high-dimensional operator $\bA$ is based on the wavelet compressibility properties of the one-dimensional operators
\begin{equation}\label{eq:1dscmatdef} 
 {\mathbf{A}}^{(i)}_2 := \hatbS_i^{-2} \bT_2 \,,\quad 
  {\mathbf{A}}^{(i)}_3 := \hatbS_i^{-1} \bT_3 = -\hatbS_i^{-1}\bT_4   %
  \,,
\end{equation}
where the last relation holds because of \eqref{T2} and the boundary conditions.
More precisely, we make use of the following property: there exist an $s>0$ and $\mathbf{T}_{n,j}$, $j\in\N$, such that
for some fixed sequences of positive numbers $\beta( {\mathbf{A}}^{(i)}_n)\in\spl{1}$ for $n=2,3$,
\begin{equation}\label{eq:diff_scompr}
\norm{\hatbS_i^{-2} (\bT_{2} -  \bT_{2,j}) } \leq \beta_j( {\mathbf{A}}^{(i)}_2) \,2^{-sj} \,,  \quad 
 \norm{ \hatbS_i^{-1} (\bT_{3} -  \bT_{3,j}) } \leq \beta_j( {\mathbf{A}}^{(i)}_3)\, 2^{-sj} \,,
 \end{equation}
where   each $\bT_{n,j}$ has at most $\alpha_j({\mathbf{A}}^{(i)}_n)\, 2^j$ nonzero entries in each column, with $\alpha( {\mathbf{A}}^{(i)}_n)\in\spl{1}$
further fixed sequences of positive numbers. It is convenient to scale the sequences so that
$\norm{\beta(\bA^{(i)}_n)}_{\spl{1}} \le \norm{ \bA^{(i)}_n }$.

Note that this is slightly weaker than the usual definition of $s^*$-compressibility \cite{Cohen:01}, since we do not require a bound on the number of entries per row, and we shall refer to the property in \eqref{eq:diff_scompr} as \emph{column}-$s^*$-\emph{compressibility}. In addition, as in \cite{BD2} we assume the approximations to have the \emph{level decay property}, that is, there exists a $\gamma>0$ such that $\abs{\abs{\nu} - \abs{\mu}} > \gamma j$ implies $T_{n,j,\nu\mu} = 0$.

Our aim is to obtain $\bw_\eta$, satisfying certain representation complexity bounds, such that $\norm{\Pc\bT \bv - \bw_\eta} \leq \eta$. We make the ansatz $\bw_\eta = \Pc_n \tbT \bv$ where $\Pc_n$ is the finite rank approximation to the scaling operator $\Pc$ from
\eqref{Pn} and $\tilde\bT$ is a ``compressed'' version of $\bT$.  
Specifically, based on the estimate
\begin{align}
   \norm{\Pc\bT \bv - \Pc_n \tbT \bv} 
     &\leq  \norm{\Pc (\bT - \tbT) \bv} +  \norm{(\Pc - \Pc_n) \tbT \bv} \notag \\
     &\leq  (1+\delta) \norm{\Sc^{-2} (\bT - \tbT) \bv }  +  \norm{(\Pc - \Pc_n) \tbT \bv},  \label{eq:applyerr}
\end{align}
we first choose $\tbT= \tbT(\bv)$ depending on $\bv$ to obtain a suitable bound on the first term on the right hand side, and subsequently pick $n$ such that the second term is sufficiently small.

The construction of $\tbT$, based on the property \eqref{eq:diff_scompr}, can be done in complete analogy to \cite[Section 4.2]{BD2}. The resulting approximation is of the form
$$ 
  \tilde{\bT} = \tilde\bT_J := \sum_{\kk{n} \in \KK{{d}}(\kk{R})} c_{\kk{n}} \bigotimes_i \tilde{\bT}^{(i)}_{n_i} \,,
$$
where $\tilde{\bT}_1^{(1)} = \bT_1 = \id$ and for $n_i > 1$,
\begin{equation}
\label{tildeTni}
\tilde{\bT}^{(i)}_{n_i} = \tilde{\bT}^{(i,J)}_{n_i} := \sum_{p=0}^{J+1} {\bT}^{(i,J)}_{n_i,[p]} \Restr{\Lambda^{(i)}_{[p]} }
\end{equation}
with $\mathbf{T}^{(i)}_{n_i,[p]} :=   \mathbf{T}_{n_i,J-p}$, $p=0,\ldots,J$,   and $\mathbf{T}^{(i)}_{n_i,[J+1]} := 0$ as in \eqref{eq:diff_scompr}. 
Recall from Section \ref{sssec:1.2.2} that the operator $\Restr{\Lambda}$ retains the entries of a tensor supported in $\Lambda$ and replaces
all others by zero. The adaptive $\bv$-dependent formation of $\tbT$ hinges on the choice of the intex sets $\Lambda^{(i)}_{[p]}$, which
 are constructed from the supports $\bar{\Lambda}^{(i)}_j$ of the best $2^j$-term approximations of $\pi^{(i)}(\bv)$.
 Specifically, setting $\bar{\Lambda}^{(i)}_{-1} := \emptyset$, we recursively define
\begin{equation*}
  \Lambda^{(i)}_{[p]} := \bar{\Lambda}^{(i)}_p\setminus \bar{\Lambda}^{(i)}_{p-1} \,,\; p=0,\ldots,J,\quad
     \Lambda^{(i)}_{[J+1]} := \nabla \setminus \bar{\Lambda}^{(i)}_J\,,\quad \Lambda^{(i)}_{[p]}:=\emptyset \,,\; p> J+1\,.
\end{equation*}
Defining next   the a posteriori error indicator
\begin{multline}
\label{eJv}
e_J(\bv) :=  \sum_{i=1}^d C^{(i)}_\bA  \Bigl[  
     \sum_{p=0}^J \Bigl(\sum_{n=2}^R \beta_{J-p}(\bA^{(i)}_n)  \Bigr) 2^{-s(J-p)}   \norm{\Restr{\Lambda^{(i)}_{[p]}} \pi^{(i)}(\bv)} 
\\
  +  \sum_{n=2}^R \norm{{\bA}^{(i)}_n} %
\,  \norm{\Restr{\Lambda^{(i)}_{[J+1]}} \pi^{(i)}(\bv)} \Bigr] ,
\end{multline}
where
\begin{equation}
\label{eq:maxsequences-0}  
    C^{(i)}_\bA := \max \Bigl\{  \abs{a_{ii}}  ,2
      \sum_{j\neq i}  \norm{{\bA}^{(j)}_3} 
     \abs{a_{ij}} ,    \Bigr\} 
           \leq \max\bigl\{ 1,2 \max_{j\neq i} \norm{\bA^{(j)}_3} \bigr\} \abs{a_{ii}},
\end{equation}
one can follow the arguments in \cite[Lemma 6.10]{BD2}, now using \eqref{eq:diff_scompr}, to verify that
\begin{equation}
\label{upperbound}
\norm{\Sc^{-2}(\bT - \tilde{\bT}_J)\bv } \le e_J(\bv).
\end{equation}

The heart of Algorithm \ref{alg:tensor_opeq_solve} is the adaptive application of $\bA$.
We can now specify the corresponding routine $\apply(\bv;\eta)$ for a finitely supported input $\bv\in \ell_2(\nabla^d)$ and a prescribed
error tolerance $\eta >0$. The relevant properties are collected in the following theorem, which is a complete analog
to Theorem 6.8 in \cite{BD2}. 

Without loss of generality, for a given $\bv$ we shall employ tolerances $\eta \le \|\bS^{-2}\bT\|\|\bv\|$, since otherwise we may choose $\bw_\eta = 0$.
For such $\eta$, it will be convenient to define 
\begin{equation}
\label{zeta}
\zeta(\eta;\bv):= \frac{\eta}{3 \norm{\bS^{-2}\bT} \norm{\bv} }.
\end{equation}

\begin{thrm}
\label{thm:apply}
Given any $\bv\in \ell_2(\nabla^d)$ of finite support and finite hierarchical ranks as well as any $0<\eta \le \|\bS^{-2}\bT\|\|\bv\|$,  let $\bw_\eta$ be defined as follows: choose $J(\eta;\bv)$ as the minimal integer such that 
\begin{equation}
\label{Jcond}
(1+\delta) e_{J(\eta;\bv)}(\bv) \leq \frac\eta2,
\end{equation}
and set $\bw_\eta := \Pc_{m(\eta;\bv)} \tilde{\bT}_{J(\eta;\bv)} \bv$ where,
with $\tilde\Lambda := \bigtimes_{i=1}^d\displaystyle \supp_i (\tilde{\bT}_{J(\eta;\bv)} \bv)$,
\begin{equation}
\label{nsize}
  m(\eta;\bv) :=  \bigl\lceil{ h^{-1} \bigl( \abs{\ln(\zeta(\eta;\bv)} + \ln \max \{  (\om{\nu}/\omin)^2 \colon  \nu \in \tilde\Lambda \} \bigr) }\bigr\rceil \,. 
  \end{equation}
Then the following statements hold:
\begin{enumerate}[{\rm(i)}]
\item We have the estimates 
\begin{align} 
\label{eq:approx-eta}
\norm{\mathbf{A}\mathbf{v} - \mathbf{\bw_\eta}} & \leq \eta\,,   \\
   \#  \supp_i (\bw_\eta)  &\leq 
  \|\hat\alpha\|_{\ell_1}  \eta^{-\frac{1}{s}}
  \Big(4(2^{s}+2) R^{1+s} \sum_{i=1}^d C^{(i)}_\bA \max_{n>1} \norm{\bA^{(i)}_n}\,   \norm{\pi^{(i)}(\bv)}_{\As}
   \Big)^{\frac1s},
   \label{eq:tensor_apply_support} 
\end{align}
where $\hat\alpha :=(\hat\alpha_k)_{k\in\N}$ and $\hat\alpha_k := \max_{i\in\{1,\ldots,d\}} \max_{n>1} \alpha_k({\bA}^{(i)}_n)$.
\item
The outputs of $\apply$ are sparsity-stable in the sense that for $ i\in\{1,\ldots,d\}$,
\begin{equation}
\label{sparsity-stable}
\norm{\pi^{(i)}( \bw_\eta )}_\As \leq \Bigl( \check{C}^{(i)}_\bA  + \frac{2^{3s+2}}{2^s - 1} \norm{\hat{\alpha}}_{\spl{1}}^s 
\max_{n > 1} \norm{\bA^{(i)}_n}\, 
C^{(i)}_\bA \Bigr) R^s (1+\delta) \,  \norm{\pi^{(i)}(\bv)}_\As \,,
\end{equation}
 where $C^{(i)}_\bA$ is defined in \eqref{eq:maxsequences-0} and 
\begin{equation}
\label{checkAi-bound}
\check{C}^{(i)}_\bA  :=  12\, (d-1) \max_{j\neq i} \abs{a_{jj}}\,  \bigl(\max_{i,n_i} \norm{\bA^{(i)}_{n_i}}\bigr)^2
 \,.
\end{equation}
\item
Denoting by $R_\alpha$ the hierarchical ranks in the representation of $\bT$,
the hierarchical ranks of  $\bw_\eta$ can be bounded by
\begin{equation}
 {\rank_\alpha (\bw_\eta  )} \leq \hat m(\eta;\bv) R_\alpha \rank_\alpha(\mathbf{v}),\quad \alpha \in \cD_d \,,
   \label{eq:tensor_apply_ranks}
\end{equation}
 where for $n^+ = n^+(\delta)$ from in Section \ref{ssec:precond} and $m(\eta;\bv)$ defined in \eqref{nsize},
\begin{equation}
\label{scrankstotal}
\hat m(\eta;\bv) := 1 +n^+ + m(\eta;\bv). 
\end{equation}
 \item
The number ${\ops}(\bw_\eta)$ of floating point operations required to compute $\bw_\eta$ in the 
hierarchical Tucker format for a given $\bv$ with ranks $\rank_\alpha(\bv)= r_\alpha$, $\alpha \in \cD_d\setminus \{0_d\}$, and $r_{\hroot{d}}=1$, scales like
\begin{multline}
\label{eq:flops}
{\ops}(\bw_\eta)\lesssim  \sum_{\alpha\in \Ncal(\hdimtree{d})}  \bigl(\hat m(\eta;\bv)\bigr)^3 R_\alpha r_\alpha \prod_{q=1}^2
R_{c_q(\alpha)}r_{c_q(\alpha)} \\
 +  \eta^{-1/s} \sum_{i=1}^d \|\hat\alpha\|_{\ell_1}  \hat m(\eta;\bv) R r_i \Big(\sum_{j=1}^d C^{(j)}_\bA R\|\pi^{(j)}(\bv)\|_{\As}\Big)^{1/s},
\end{multline}
where the constant is independent of $\eta, \bv$, and $d$.
\item
Assume in addition that the approximations $\bT_{n,j}$ have the level decay property. With the notation $L(\bv):= \max\{ \abs{\nu_i} \colon \nu_i\in\supp_i( \bv),\,i=1,\ldots,d \}$, the scaling ranks $\hat m(\eta;\bv)$, defined in \eqref{scrankstotal}, can be bounded by
\begin{equation}
\label{metav}
\hat m(\eta;\bv) \le C(\delta, s,\bA)\,\Big[1 + L(\bv)  +  \abs{\ln \eta } +   \ln \Big(\sum_{i=1}^d
 \norm{\pi^{(i)}(\bv)}_\As\Big) \Big].
\end{equation}
\end{enumerate}
\end{thrm}

Comparing the above statements with Theorem 6.8 in \cite{BD2} reveils several minor differences. This concerns, for instance,
the constants in \eqref{eq:tensor_apply_support}, with the condition \eqref{Jcond}
is slightly relaxed here, and the somewhat less involved definition of $m(\eta;\bv)$ in \eqref{nsize} due to the one-sided
application of the scaling operator. The main difference lies in the rank bounds \eqref{eq:tensor_apply_ranks} and in the 
bound on the number of operations \eqref{eq:flops}, where $\hat m(\eta;\bv)$ enters with half the exponent of \cite[Theorem 6.8]{BD2}.

The proof of Theorem \ref{thm:apply} differs from the proof of Theorem 6.8 in \cite{BD2} only in minor technical details. %
In fact, the one-sided scaling simplifies some of the arguments. We therefore give some brief comments and omit a complete proof.

First, with $\tilde\Lambda$ as in Theorem \ref{thm:apply}, one has
\[  \norm{(\Pc - \Pc_{m(\eta;\bv)}) \tbT \bv}  \leq 
   \norm{(\Pc - \Pc_{m(\eta;\bv)}) \Sc^2 \Restr{\tilde\Lambda}  } \bigl(  e_{J(\eta;\bv)}  
      +  \norm{\Sc^{-1} \bT}  \norm{\bv}  \bigr)  \,.
   \]
Combining this with \eqref{eq:applyerr}, \eqref{upperbound}, and \eqref{nsize} yields
$$
\norm{ \bA \bv - \bw_\eta }  \le (1+\delta)e_{J(\eta;\bv)}(\bv) + \zeta(\eta;\bv)\big(e_{J(\eta;\bv)}(\bv)+ \|\bS^{-2}\bT\|\|\bv\|\big).
$$
In view of \eqref{Jcond}, $\eta \leq  \|\bS^{-2}\bT\|\|\bv\|$, and \eqref{zeta}, this confirms \eqref{eq:approx-eta}.
The argument for \eqref{eq:tensor_apply_support} is the same as in \cite{BD2}. The slightly different constant results from the 
relaxed requirement \eqref{Jcond} on $J(\eta;\bv)$. The appearance of the factor $(1+\delta)$ in \eqref{sparsity-stable}
instead of $(1+\delta)^2$ in \cite[Theorem 6.8]{BD2} results again from the one-sided scaling, which also leads to the more favorable exponents
in \eqref{eq:tensor_apply_ranks} and \eqref{eq:flops}.

\subsection{Complexity Estimates}\label{sec:complexity}
We have seen that Algorithm \ref{alg:tensor_opeq_solve} converges without any specific
assumptions on the solution in the sense that a given target accuracy is reached 
after finitely many steps. We will show next that, under canonical assumptions
on the problem data ($\bA, \bbf$), whenever the solution has certain sparsity properties
(regarding low-rank approximability and representations sparsity of the tensor factors), 
the approximate solution produced by Algorithm \ref{alg:tensor_opeq_solve} has
similar and in a sense near-optimal sparsity properties. We proceed now formulating
our data assumptions as well as the envisaged {\em benchmark assumptions} concerning
the solution. We stress, however, that these assumptions are {\em not} explicitly used by the algorithm,
but rather exploited automatically.

From the results in \cite{BD2} and Theorem \ref{thm:S2T}, we know that the infinite matrices
$\bS^{-2}\bT$ and $\bS^{-1}\bT\bS^{-1}$ are automorphisms of $\ell_2(\nabla^d)$.
In particular, $\hatbS_i^{-2}\bT_2$ and $\hatbS_i^{-1}\bT_2\hatbS_i^{-1}$ are bounded mappings on $\ell_2(\nabla)$.
  This latter fact can be interpreted as follows. Let $\ell_2^t(\nabla)$ denote the {\em weighted}
space $\{ \bw\in \R^{\nabla}: \| \hatbS^t\bw\| <\infty\}$, which defines a scale of interpolation spaces.
Then, the boundedness of $\hatbS_i^{-1}\bT_2\hatbS_i^{-1}$
means that $\hatbS_i^{-2}\bT_2 : \ell_2^1(\nabla)\to \ell_2^1(\nabla)$ is bounded. By interpolation,
$\hatbS_i^{-2}\bT_2 : \ell_2^t(\nabla)\to \ell_2^t(\nabla)$  is bounded for $t\in [0,1]$. This, in turn, means
that $\hatbS^{t-2}_i\bT_2\hatbS^{-t}_i$ is bounded for $t\in [0,1]$, and by the same argument we obtain also that
$\hatbS^{t-1}_i\bT_3\hatbS^{-t}_i$ is bounded.
Hence, for $t\in [0,1]$,
\begin{equation}
\label{Sobstab}
\norm{\hatbS^{t}_i\bA^{(i)}_2\hatbS^{-t}_i},\; \norm{\hatbS^{t}_i \bA^{(i)}_3 \hatbS^{-t}_i } < \infty
\quad \text{as well as} \quad \norm{\Sc^t \bbf} \leq (1+\delta) \norm{\Sc^{t-2} \bg} < \infty .
\end{equation}
The \emph{excess regularity} assumption made in \cite{BD2} corresponds to the statement that \eqref{Sobstab} holds for \emph{some} $t>0$, which there indeed had to be assumed. As shown by the above considerations, however, this is in our present setting automatically satisfied for $t=1$.

We now formulate our data assumptions.

\begin{assumptions}
\label{ass:A-f}
Concerning the scaled matrix representation $\bA$ and the right hand side $\bbf$ we require the following properties for some fixed $s^* > 0$:
\begin{enumerate}[{\rm(i)}]
 \item The lower-dimensional component operators $\bA^{(i)}_{n_i}$,  defined in \eqref{eq:1dscmatdef}, are column-$s^*$-compressible with the level decay property (cf.\ Section \ref{ssec:opappr}). %
 \item The number of operations required for evaluating each entry in the approximations $\bT_{n,j}$ as in \eqref{eq:diff_scompr} is uniformly bounded.
\item We have an estimate $c_\bA \geq \norm{\bA^{-1}}$, and the initial error estimate $\varepsilon_0$ overestimates the true value of $\norm{\bA^{-1}}\norm{\bbf}$ only up to some absolute multiplicative constant, i.e., $\varepsilon_0  \lesssim  \norm{\bA^{-1}}\norm{\bbf}$.
\item
The contractions of $\bbf$ are compressible, i.e., $\pi^{(i)}(\bbf)\in \As$, $i=1,\ldots,d$, for any $s$ with $0 < s
 <s^*$.
\end{enumerate}
\end{assumptions}

The concrete realization of the routine $\rhs$ depends on the concrete way the right hand side is given. 
For details on possible constructions of $\rhs$, we refer to \cite[Appendix B]{BD2}, which justifies the following
assumptions made in subsequent complexity statements.

\begin{assumptions}
\label{ass:rhs}
The procedure $\rhs$ is assumed to have  the following properties:
\begin{enumerate}[{\rm(i)}]
\setcounter{enumi}{4}
 \item \label{ass:rhsapprox}There exists an approximation $\bbf_\eta := \rhs(\eta)$ such that $\| \mathbf{f}-\rhs(\eta)\|\leq \eta$ and
  \begin{gather*}
  \norm{\pi^{(i)}(\bbf_\eta)}_\As \leq C^\text{{\rm sparse}} \norm{\pi^{(i)}(\bbf)}_{\As}    \,,  \quad   \norm{\Sc_i \bbf_\eta} \leq C^\text{{\rm reg}} \norm{\Sc_i \bbf} \,, \\
   \sum_i \#\supp_i(\bbf_\eta) \leq C^{\text{{\rm supp}}} \,d \,\eta^{-\frac1s}\, \Bigl(\sum_i \norm{\pi^{(i)}(\bbf)}_{\As}\Bigr)^{\frac1s}, \\  \abs{\rank(\bbf_\eta)}_\infty \leq   C_\bbf^{\text{{\rm rank}}}\, \abs{\ln \eta}^{b_\bbf} \,,
  \end{gather*}
hold,  where $C^\text{{\rm sparse}},C^{\text{{\rm supp}}}, C^{\text{{\rm reg}}}, C_\bbf^{\text{{\rm rank}}} > 0$,  $b_\bbf \geq 1$ are  independent of $\eta$, and $C^\text{{\rm sparse}}$, $C^{\text{{\rm reg}}}$, $C^{\text{{\rm supp}}}$ are independent of $\bbf$.
 \item \label{ass:rhsops}The number of operations required for evaluating $\rhs(\eta)$ is bounded, with a constant $C^\text{{\rm ops}}_\bbf(d)$, by
  $\ops(\bbf_\eta) \leq C^\text{{\rm ops}}_\bbf(d) \bigl[ \abs{\ln\eta}^{3b_\bbf} 
    + \abs{\ln \eta}^{b_\bbf}  \eta^{-\frac1s} \bigr]  $. 
\end{enumerate}
\end{assumptions}

Next we explain the benchmark properties of the solution to which subsequent complexity statements refer. These properties
are {\em not} used by the solver.
 
\begin{assumptions}\label{ass:approximability}
Concerning the approximability of the solution $\bu$, we   assume:
\begin{enumerate}[{\rm(i)}]
\setcounter{enumi}{6}
\item \label{ass:uapprox}$\bu \in \AH{\ga_\bu}$ with $\ga_\bu(n) = e^{d_\bu n^{1/b_\bu}}$
for some $d_\bu>0$, $b_\bu \geq 1$.
\item $\pi^{(i)}(\bu) \in \As$ for $i=1,\ldots,d$, for any $s$ with $0 < s
 <s^*$.
\end{enumerate}
\end{assumptions}

When discussing tractability issues in the sense of complexity theory it is important to know how the data
behave with respect to the spatial dimension $d$.

\begin{assumptions}\label{ass:dim}
In our comparison of problems for different values of $d$, we assume:
\begin{enumerate}[{\rm(i)}]
\setcounter{enumi}{8}
 \item The following constants are independent of $d$: $d_\bu$, $b_\bu$, $C^\text{{\rm sparse}}$, $C^{\text{{\rm supp}}}$, $C^{\text{{\rm reg}}}$, $C_\bbf^{\text{{\rm rank}}}$.
  \item The following quantities remain bounded independently of $d$: $\norm{\bA}$, $\norm{\bA^{-1}}$; the maximum hierarchical representation rank $\max_\alpha R_\alpha$ of $\bT$;
  the quantities $\norm{\pi^{(i)}(\bu) }_\As$ in the benchmark assumptions, $\norm{\pi^{(i)}(\bbf)}_{\As}$ in Assumptions \ref{ass:approximability}\eqref{ass:rhsapprox}, each for $i=1,\ldots,d$.
\item  In addition, we assume that  $C^\text{{\rm ops}}_\bbf(d)$ as in Assumptions \ref{ass:approximability}\eqref{ass:rhsops} grows at most polynomially as $d\to \infty$.
\item There exists a choice of $\tilde{\rho}$ in \eqref{K} independent of $d$ such that the corresponding values $K(\tilde{\rho})$, $B(\tilde{\rho})$ are bounded independently of $d$ as well.
 \end{enumerate}
 \end{assumptions}
 
Concerning the assumptions on $\norm{\bA}$, $\norm{\bA^{-1}}$, see Theorem \ref{thm:S2T} in Section \ref{ssec:ddep}. 
Concerning (xii), we know from the discussion in Section \ref{ssec:convan} that the existence of $K, B$ as in \ref{K} is ensured for $\tilde{\rho} > \rho$. Since the values of corresponding $K$ and $B$ are not explicitly quantified, however, the a priori bound on the number of steps \eqref{eq:Idef} serves only for theoretical purposes, and we have to rely on an a posteriori condition on the approximate residual for controlling the iteration. The concrete resulting values of $K$ and $B$ may depend on the choice of basis functions. As our numerical examples demonstrate, these values do not have any significant influence in practice.

The main result of this paper reads as follows. 
 
\begin{thrm}
\label{thm:complexity}
Suppose that  Assumptions \ref{ass:A-f}, \ref{ass:rhs} hold and that Assumptions \ref{ass:approximability} are valid for the solution $\bu$  of $\mathbf{A}\bu = \mathbf{f}$.
Let $\alpha > 0$ and let $\constsvd, \constcrs$ be as in Theorem
\ref{lmm:combined_coarsening}.
Let the constants $\kappa_1,\kappa_2,\kappa_3$ in
Algorithm \ref{alg:tensor_opeq_solve} be chosen as
\begin{gather*}
  \kappa_1 = \bigl(1 + (1+\alpha)(\constsvd + \constcrs +
  \constsvd\constcrs)\bigr)^{-1}\,, \\
  \kappa_2 = (1+\alpha)\constsvd \kappa_1\,,\qquad 
  \kappa_3 = \constcrs(\constsvd + 1)(1+\alpha)\kappa_1 \,,
\end{gather*}
and let $\beta_1 \geq 0$, $\beta_2 >0$ be arbitrary but fixed.
Then the approximate solution $\bu_\varepsilon$ produced by Algorithm \ref{alg:tensor_opeq_solve} for $\varepsilon < \varepsilon_0$
satisfies
\begin{gather}
 \label{eq:complexity_rank} 
 \abs{\rank(\bu_\varepsilon)}_\infty 
   \leq \, \bigl( d_\mathbf{\bu}^{-1}
   \ln\bigl[ 2(\alpha \kappa_1)^{-1} \rho_{\ga_\bu}
   \,\norm{\bu}_{\AH{\ga_\mathbf{\bu}}}\,\varepsilon^{-1}\bigr] \bigr)^{ b_\mathbf{\bu}}
    \lesssim (\abs{\ln \varepsilon} + \ln d)^{b_\bu} \,,
   \\
 \label{eq:complexity_supp} \sum_{i=1}^d \#\supp_i(\bu_\varepsilon) \lesssim
     d^{1 + s^{-1}} \, \Bigl(\sum_{i=1}^d \norm{ \pi^{(i)}(\bu)}_{\As} \Bigr)^{\frac{1}{s}}
           \varepsilon^{-\frac{1}{s}} \,,
\end{gather}
as well as
\begin{gather}
  \label{eq:complexity_ranknorm} 
  \norm{\bu_\varepsilon}_{\AH{\ga_\mathbf{\bu}}}
  \lesssim \sqrt{d}\,
      \norm{\bu}_{\AH{\ga_\mathbf{\bu}}}    \,,   \\
  \label{eq:complexity_sparsitynorm} \sum_{i=1}^d \norm{
  \pi^{(i)}(\bu_\varepsilon)}_{\As} \lesssim d^{1 + \max\{1,s\}} 
      \sum_{i=1}^d \norm{ \pi^{(i)}(\bu)}_{\As}  \,.
\end{gather}
The multiplicative constant in \eqref{eq:complexity_ranknorm} depends only on $\alpha$,
those in \eqref{eq:complexity_supp} and \eqref{eq:complexity_sparsitynorm} depend only on
$\alpha$ and $s$.

If in addition, Assumptions \ref{ass:dim} hold, then for the number of required operations $\ops(\bu_\varepsilon)$, we have the estimate
\begin{equation} 
\label{eq:complexity_totalops}
\ops(\bu_\varepsilon) \leq Cd^a \,d^{c s^{-1} \ln d} d^{12 c \ln \ln d}
   \abs{\ln \varepsilon}^{2 c \ln d + 2 \max\{b_\bu, b_\mathbf{f}\}}\,
       \varepsilon^{-\frac{1}{s}} \,,
\end{equation}
where $C,a$ are constants independent of $\varepsilon$ and $d$, and $c$ is the smallest $d$-independent value such that $I \leq c \ln d$ for $I$ as in \eqref{eq:Idef}. In particular,  $c$   does not depend on $\varepsilon$ and $s$.
\end{thrm}

As in \cite{BD2}, the proof of Theorem \ref{thm:complexity} has two main constituents. On the one hand, one can use Theorem \ref{thm:apply}
in complete analogy to the use of Theorem 6.8 in \cite{BD2}. On the other hand, one has to control $L(\bv)$ in \eqref{metav}.
On account of \eqref{Sobstab}, this can be done exactly as in \cite[Sections 6.4, 6.5]{BD2}. 

While the theoretical bounds have the same structure as for the scheme in \cite{BD2}, the concrete values of the constants are different and in fact more favorable (mainly due to the smaller exponents in \eqref{eq:tensor_apply_ranks} and \eqref{eq:flops}), as shown also by the numerical experiments discussed in the next section.

\section{Numerical Realization}\label{sec:numre}

\subsection{Approximate Application of Operators}

We now describe some practical improvements for the approximate application of operators in low-rank form required in Algorithm \ref{alg:tensor_opeq_solve}. Recall that for given compactly supported $\bv$ and tolerance $\eta>0$, we determine a suitable approximation $\tbT$ of $\bT$ as well as an $n$ such that $\norm{\Pc\bT \bv - \Pc_n \tbT \bv} \leq \eta$.

For our complexity estimates, we have assumed the choice of the parameter $n$ to be based directly on Theorem \ref{thm:expsum_relerr}. This choice depends only on $\eta$ and on the maximum wavelet level in the support of $\bv$, that is, on $\max_{\nu\in\supp \bv} \max_i \abs{\nu_i} $. We may, however, use the estimates in Theorem \ref{thm:expsum_relerr} in a slightly different way to take the actual values of $\bv$ into account, and hence make use of additional a posteriori information.

According to \eqref{eq:applyerr} we first choose, independently of $n$, a suitable $\tbT$ such that
$(1+\delta) \norm{\Sc^{-2} (\bT - \tbT) \bv } \leq \frac\eta2$. It then remains to pick $n$ such that $\norm{(\Pc - \Pc_n) \tbT \bv} \leq \frac\eta2$; here we can simply take into account the concrete values of $ \tbT \bv$ by noting that
\[
  \norm{(\Pc - \Pc_n) \tbT \bv} \leq \max_{\nu} \abs{p_\nu - p_{n,\nu}} \norm{\tbT \bv} \,.
\]
In view of \eqref{eq:ptail}, it thus suffices to take 
\[ 
 n = \left\lceil h^{-1} \biggabs{\ln\biggl(\frac{\omin^2\eta}{\norm{\tbT \bv}}\biggr) } \right\rceil \,.
\]
This choice of $n$ is typically substantially smaller than the theoretical upper bounds in Theorem \ref{thm:apply}, where we needed to take additional measures to bound $\norm{\tilde{\bT}\bv}$ and hence started instead from an estimate of the form $\norm{(\Pc - \Pc_n) \tbT \bv} \leq \norm{(\Pc - \Pc_n)\bS^2\Restr{\supp\tilde{\bT}\bv}} \norm{\bS^{-2} \tilde{\bT}\bv}$.

For the evaluation of $\Pc_n\tbT\bv$, we additionally use a scheme analogous to the one described in \cite[Section 7.2]{BD2} to add terms incrementally with additional tensor truncations, but preserving the total accuracy tolerance. To this end, we adjust the approximate operator evaluation such that $\Pc_n \tbT =: \bw_{\eta/2}$ satisfies $\norm{\Pc\bT \bv - \bw_{\eta/2}} \leq \eta/2$, and then determine an approximation $\tilde\bw_{\eta/2}$ with $\norm{\bw_{\eta/2} - \tilde\bw_{\eta/2}} \leq \eta/2$, which is subsequently used as the output of $\apply(\bv;\eta)$. 
With $\Pc_n = \sum_{\ell=1}^{\hat m(n)} \Theta_\ell$ and $\mathbf{\tilde t} := \tbT \bv$, we first evaluate $\tau_\ell := \norm{\Theta_\ell \mathbf{\tilde t}}$ for each $\ell$, build the ascendingly sorted sequence $\hat\tau_q := \tau_{\ell(q)}$, and find $q_0$ such that $\sum_{q=1}^{q_0} \hat\tau_q \leq \eta/4$. The remaining contributions $\Theta_{\ell(q)}\mathbf{\tilde t}$ for $q=q_0+1,\ldots,\hat m(n)$ are then summed in increasing order, with an application of $\recompress(\cdot; \zeta_q)$ after adding each summand, with $\sum_{q=q_0+1}^{\hat m(n)}\zeta_q \leq \eta/4$. At this point, we deviate slightly from the treatment in \cite{BD2}, and choose $\zeta_q$ using a posteriori information: as a by-product of $\recompress(\cdot; \zeta_q)$, we obtain an estimate $\tilde\zeta_q$ of the actual truncation error, where usually $\tilde \zeta_q < \zeta_q$. To make use of this, we set $\tilde\eta_{q_0+1}:=\eta/4$, and for each $q \geq q_0+1$ take $\zeta_q := \tilde\eta_q \hat\tau_q / \sum_{p=q}^{\hat m(n)} \hat\tau_p$ and $\tilde \eta_{q+1} := \eta_q - \tilde\zeta_q$. In this manner, truncation tolerances are again assigned in dependence on the relative sizes of summands. 

\subsection{Numerical Experiments}

In our numerical tests, we first treat the same high-dimensional Poisson problem as in \cite{BD2} to allow a direct comparison to the algorithm
with convergence enforced in $\spH{1}$-norm that we considered there. Subsequently, we apply the new scheme to a problem with tridiagonal diffusion matrix $M$. As in \cite{BD2}, we use $\spL{2}$-orthonormal, continuously differentiable, piecewise polynomial Donovan-Geronimo-Hardin multiwavelets \cite{DGH:99} of polynomial degree 6 and approximation order 7, which satisfy the conditions mentioned in Remark \ref{remark:H2rbcond} and thus form a Riesz basis of $\spH{2}(0,1)\cap \spH{1}_0(0,1)$ after rescaling.

\subsubsection{High-Dimensional Poisson Problem}

Figures \ref{fig:res}, \ref{fig:ranks}, and \ref{fig:ops} show the results for the Poisson problem on $(0,1)^d$. In comparison to the results obtained in \cite{BD2},
we generally observe a similar behavior, with the expected residual reduction and with ranks increasing gradually as the accuracy increases. The computational simplifications in the new scheme are apparent in Figure \ref{fig:ops}: with similar operation counts and error bounds, we can now go up to $d=256$ instead of $d=64$. 
However, the price to pay is that all error estimates now correspond to the $\spL{2}$-norm, instead of the $\spH{1}$-norm as in \cite{BD2}.
As illustrated in Figure \ref{fig:errestcomp}, where we compare $\spL{2}$- and $\spH{1}$-errors to a reference solution computed by a highly accurate exponential sum approximation \cite{Grasedyck:04,Hackbusch:05}, we indeed no longer have control over the error in $\spH{1}$ in the present case, but do obtain an upper bound for the $\spL{2}$-error as guaranteed by our theory.
\begin{figure}[tp]
\centering
\begin{tabular}{ccc}\hspace{-.6cm}
\includegraphics[width=4.3cm]{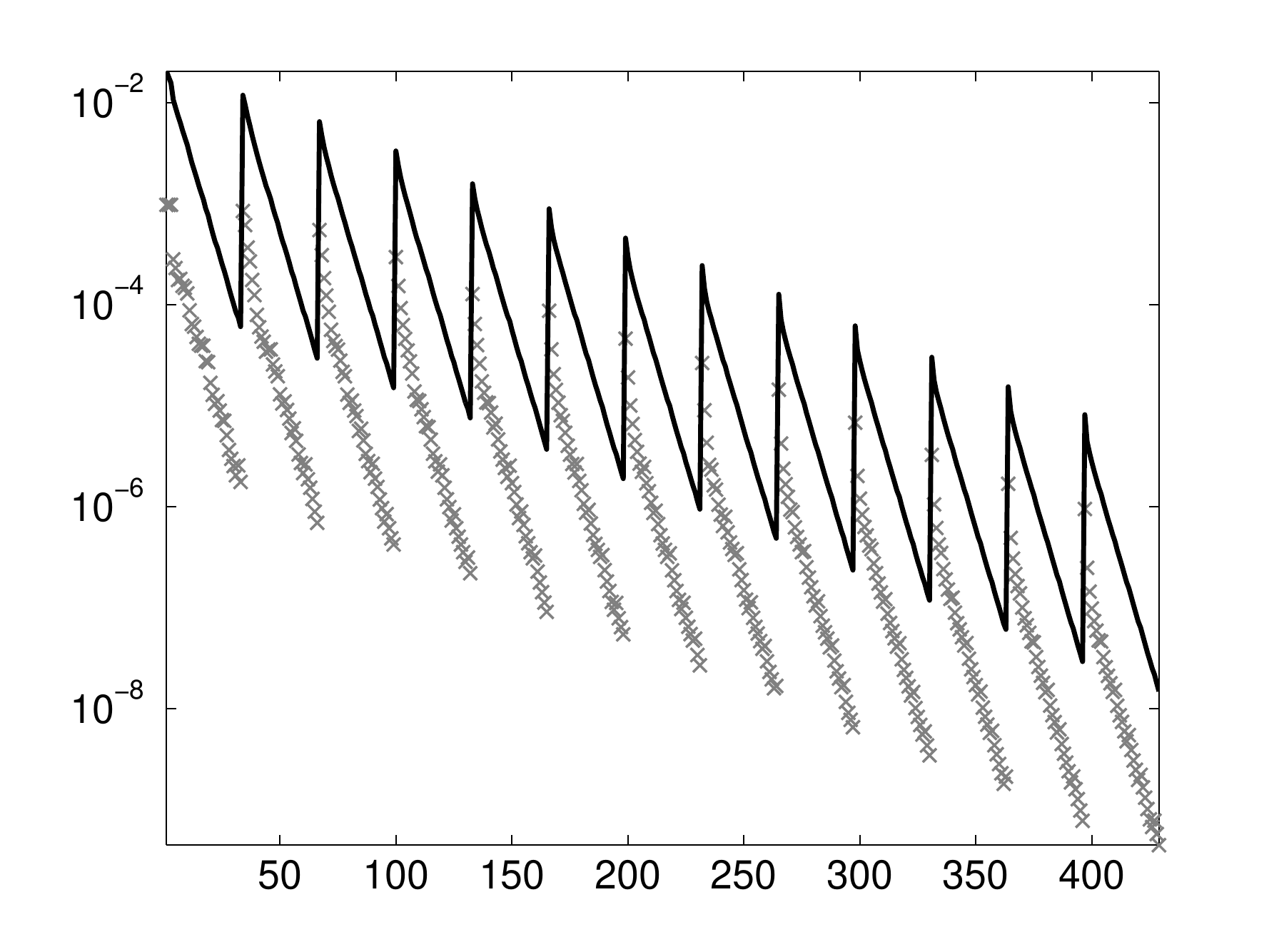} &
\includegraphics[width=4.3cm]{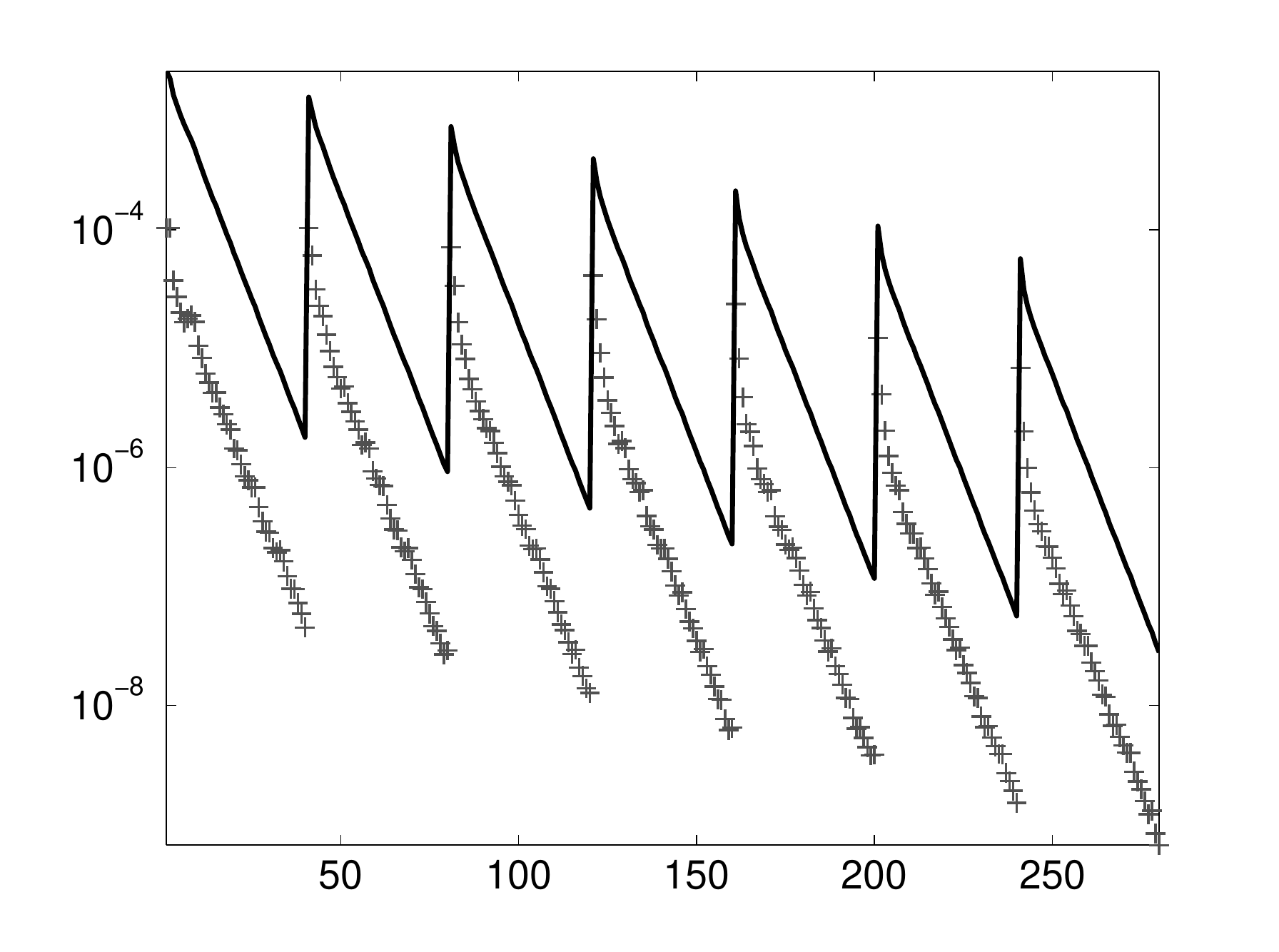} &
\includegraphics[width=4.3cm]{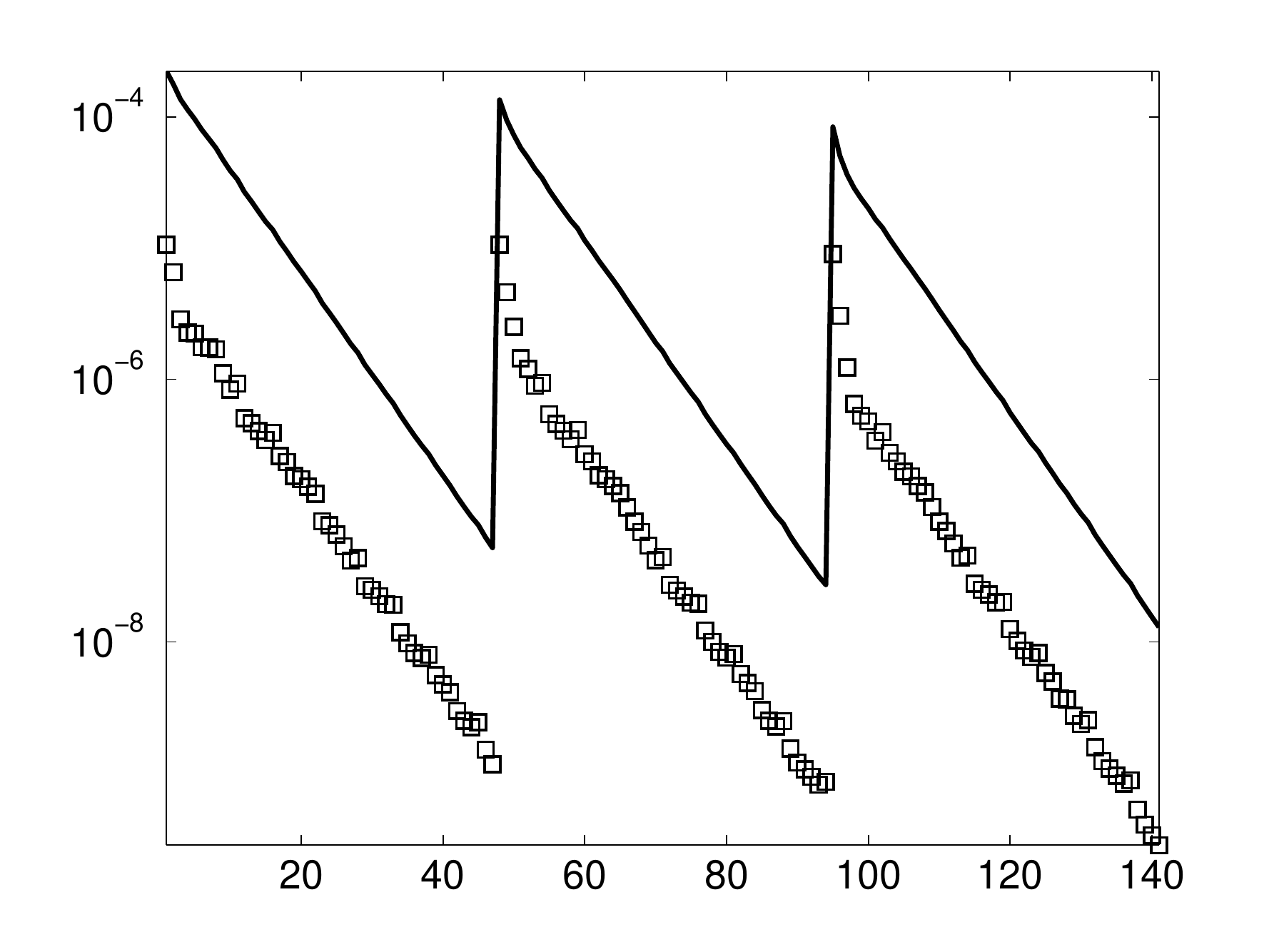}\hspace{-.5cm}
\end{tabular}
\caption{Norms of computed residual estimates (markers) and corresponding error bounds (lines),
in dependence on the total number of inner iterations (horizontal axis), for  $d=\textcolor{lightgray}{\boldsymbol{\times}} 16,\textcolor{gray}{\boldsymbol{+}} 64, \square 256$.}
\label{fig:res}
\end{figure}

\begin{figure}[tp]
\centering
\begin{tabular}{cc}\hspace{-1cm}
\includegraphics[width=6cm]{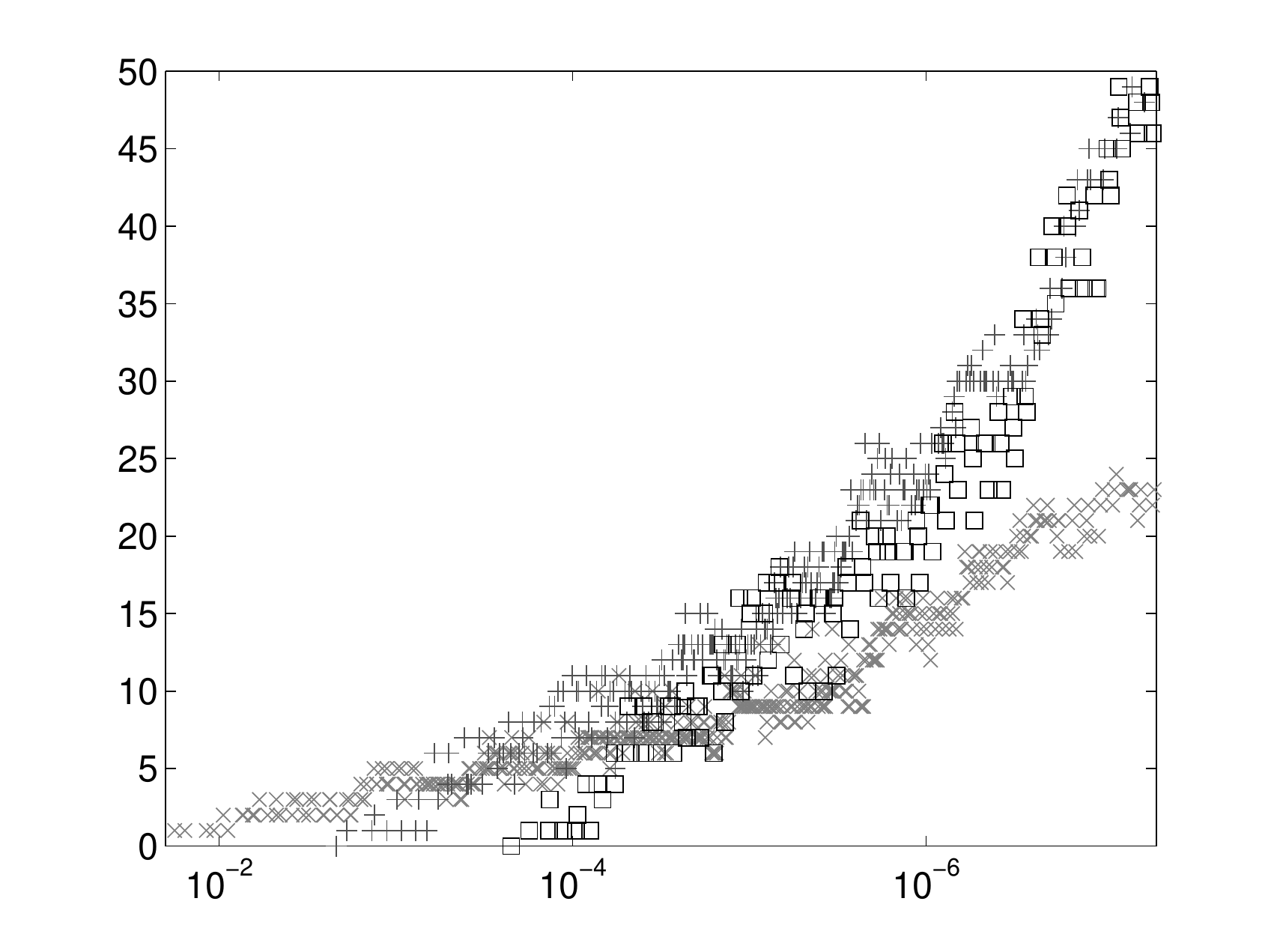} &
\includegraphics[width=6cm]{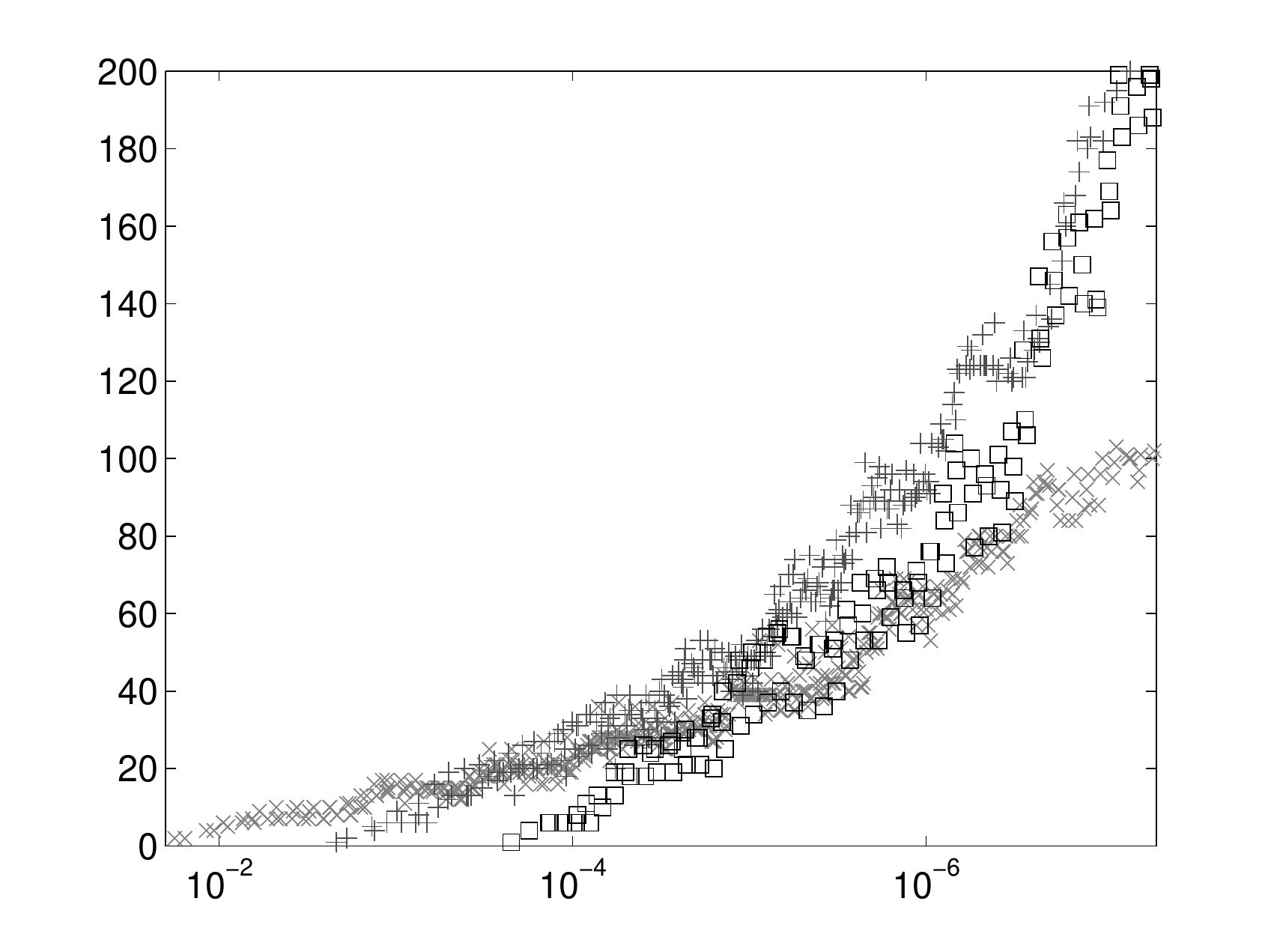}
\end{tabular}
\caption{$\abs{\rank(\bw_{k,j})}_\infty$ (left) and maximum ranks of all intermediates arising in the inner iteration steps (right),
in dependence on current estimate for $\norm{\bu - \bw_{k,j}}$ (horizontal axis), for  $d=\textcolor{lightgray}{\boldsymbol{\times}} 16,\textcolor{gray}{\boldsymbol{+}} 64, \square 256$.}
\label{fig:ranks}
\end{figure}

\begin{figure}[tp]
\centering
\includegraphics[width=8cm]{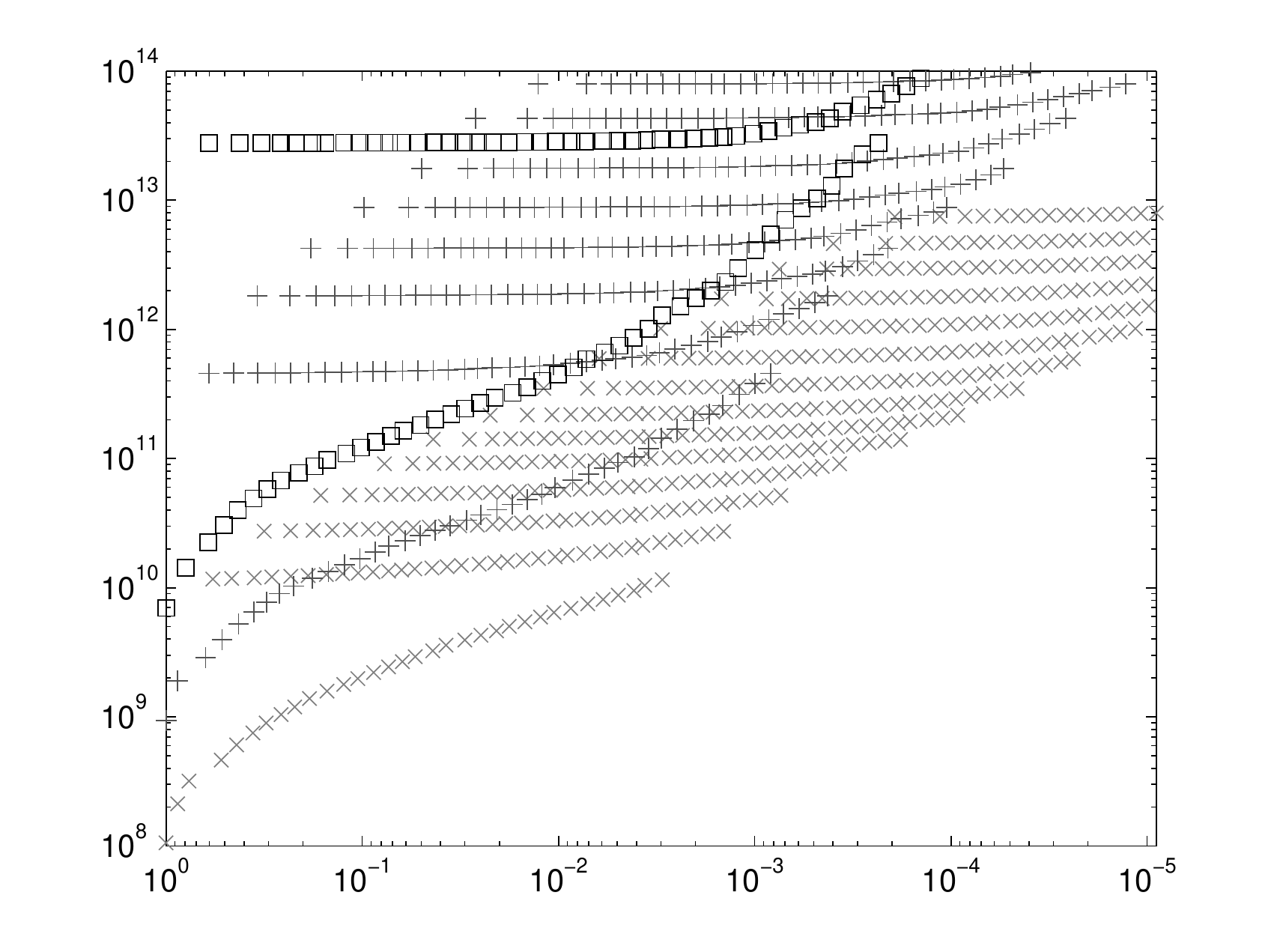}
\caption{Operation count in dependence on the error estimate reduction (horizontal axis), for $d=\textcolor{lightgray}{\boldsymbol{\times}} 16,\textcolor{gray}{\boldsymbol{+}} 64, \square 256$.}
\label{fig:ops}
\end{figure}

\begin{figure}[tp]
\centering
\includegraphics[width=8cm]{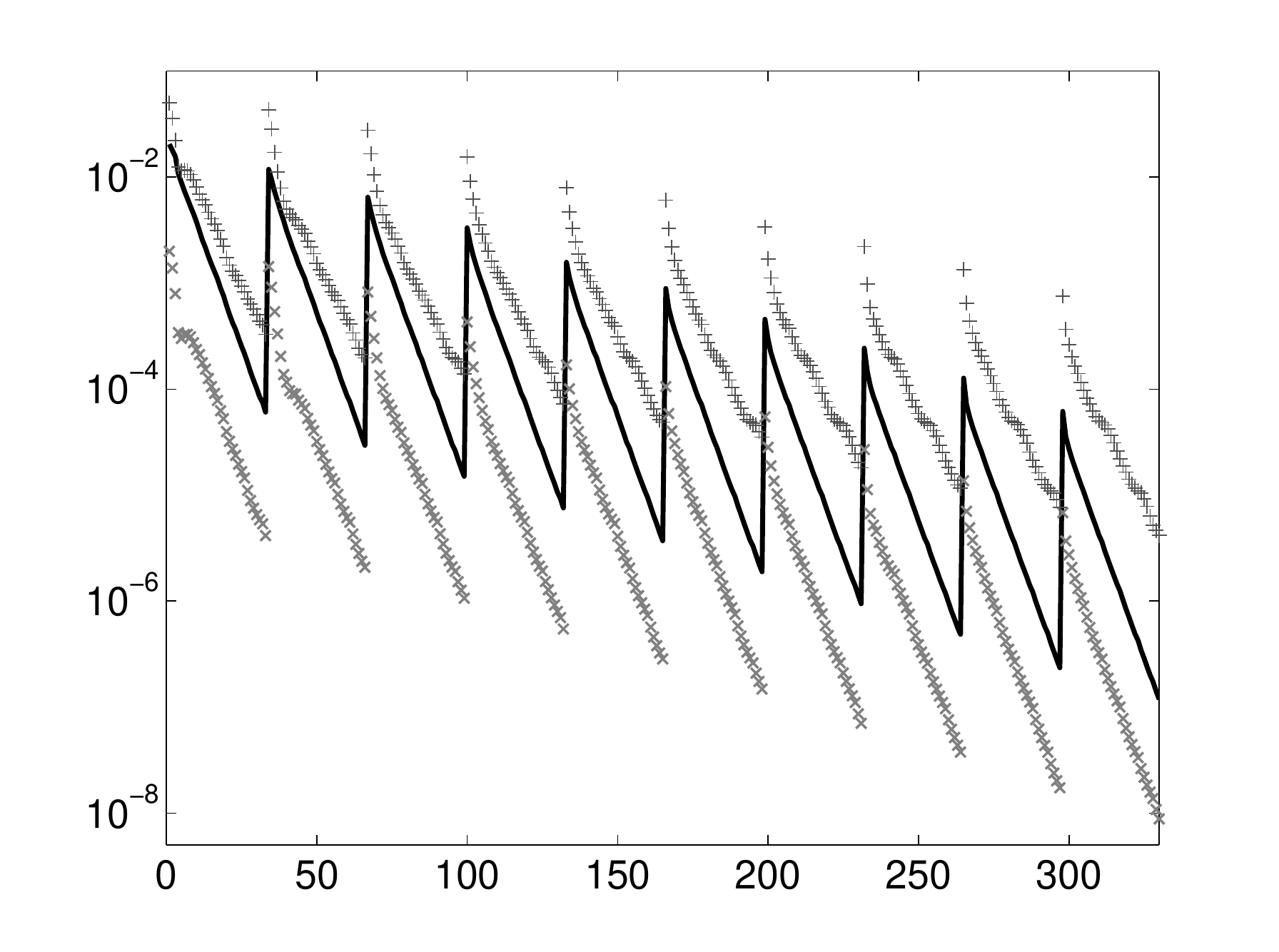}
\caption{Computed error bound (lines), differences in $\spL{2}$ ($\textcolor{lightgray}{\boldsymbol{\times}}$) and $\spH{1}$ ($\textcolor{gray}{\boldsymbol{+}}$) to reference solution,
in dependence on the total number of inner iterations (horizontal axis), for  $d= 16$.}
\label{fig:errestcomp}
\end{figure}

\subsubsection{Dirichlet Problem with Tridiagonal Diffusion Matrix}

We now consider the case of tridiagonal diffusion matrices 
\[  M=(m_{ij})_{i,j=1,\ldots,d} = \operatorname{tridiag}(-a, 2,-a)  \]
for $a=\frac12$ and $a=1$. 
As noted in \cite[Section 7.4]{BD2}, there is a significant difference in the behavior of the iteration and in the expected tensor approximability of the
solution for these two values of $a$, since for $0\leq a<1$, the condition number of $\Sc^{-1}\bT\Sc^{-1}$ (which directly affects the lower bound for $\tilde\rho$ in the present scheme) remains bounded independently of $d$, whereas it grows proportionally to $d^2$ for $a=1$. Figure \ref{fig:trir4comp} shows how this fact already manifests itself in a pronounced difference in the respective solution ranks observed for $d=4$. The $d$-dependent condition number for $a=1$ also leads to a substantial deterioration in the convergence of the iteration as $d$ increases. For $a=\frac12$, however, we are still able to treat large values of $d$, as shown in Figures \ref{fig:triranks} and \ref{fig:triops}.
\begin{figure}[tp]
\centering
\includegraphics[width=8cm]{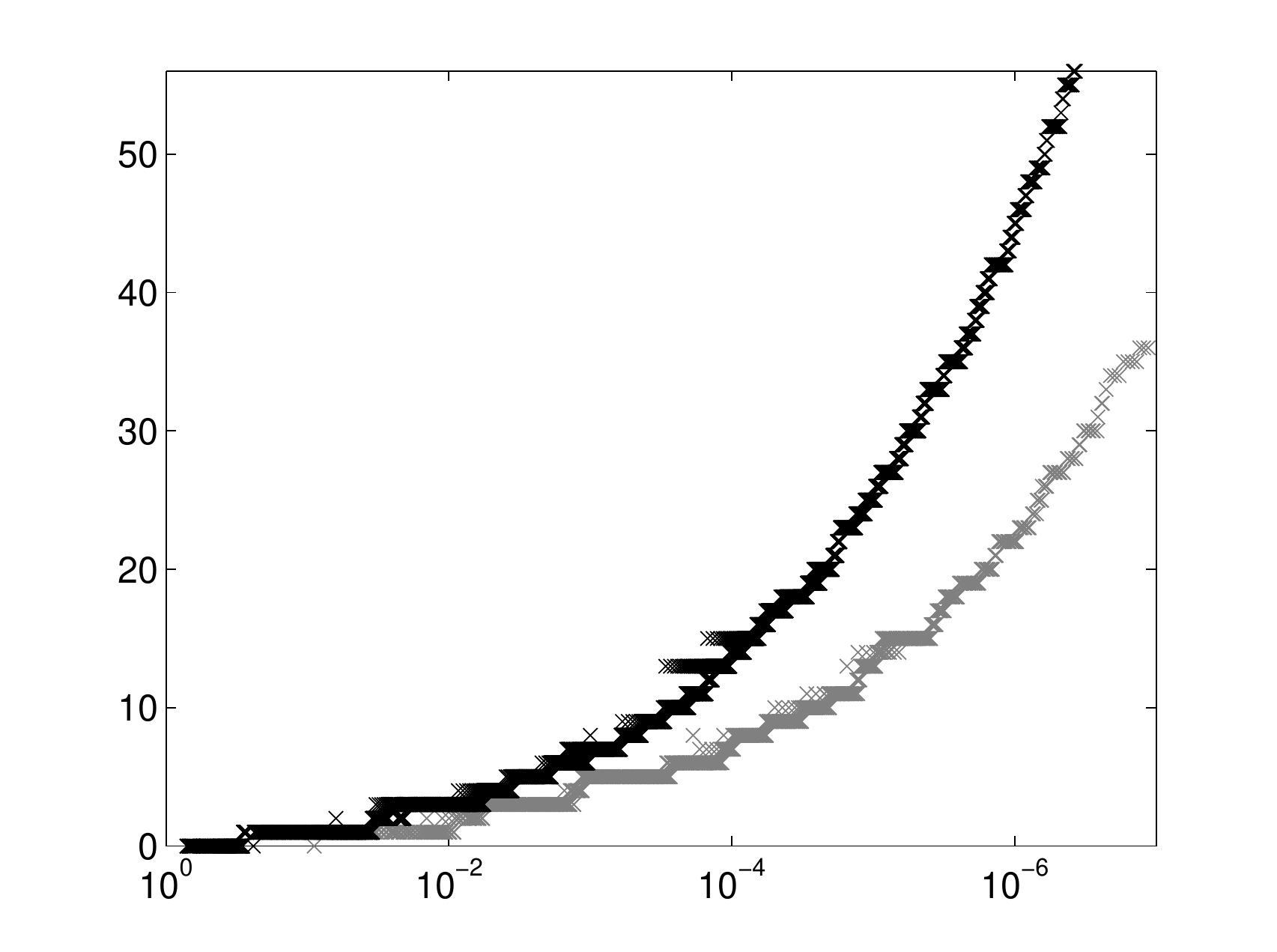}
\caption{Tridiagonal diffusion matrix, $\abs{\rank(\bw_{k,j})}_\infty$ in dependence on current estimate for $\norm{\bu - \bw_{k,j}}$ (horizontal axis), for $d=4$ and $a = \frac12$ ($\textcolor{gray}{\boldsymbol{\times}}$) and $a=1 $ ($\boldsymbol{\times}$).}
\label{fig:trir4comp}
\end{figure}

\begin{figure}[tp]
\centering
\begin{tabular}{cc}\hspace{-1cm}
\includegraphics[width=6cm]{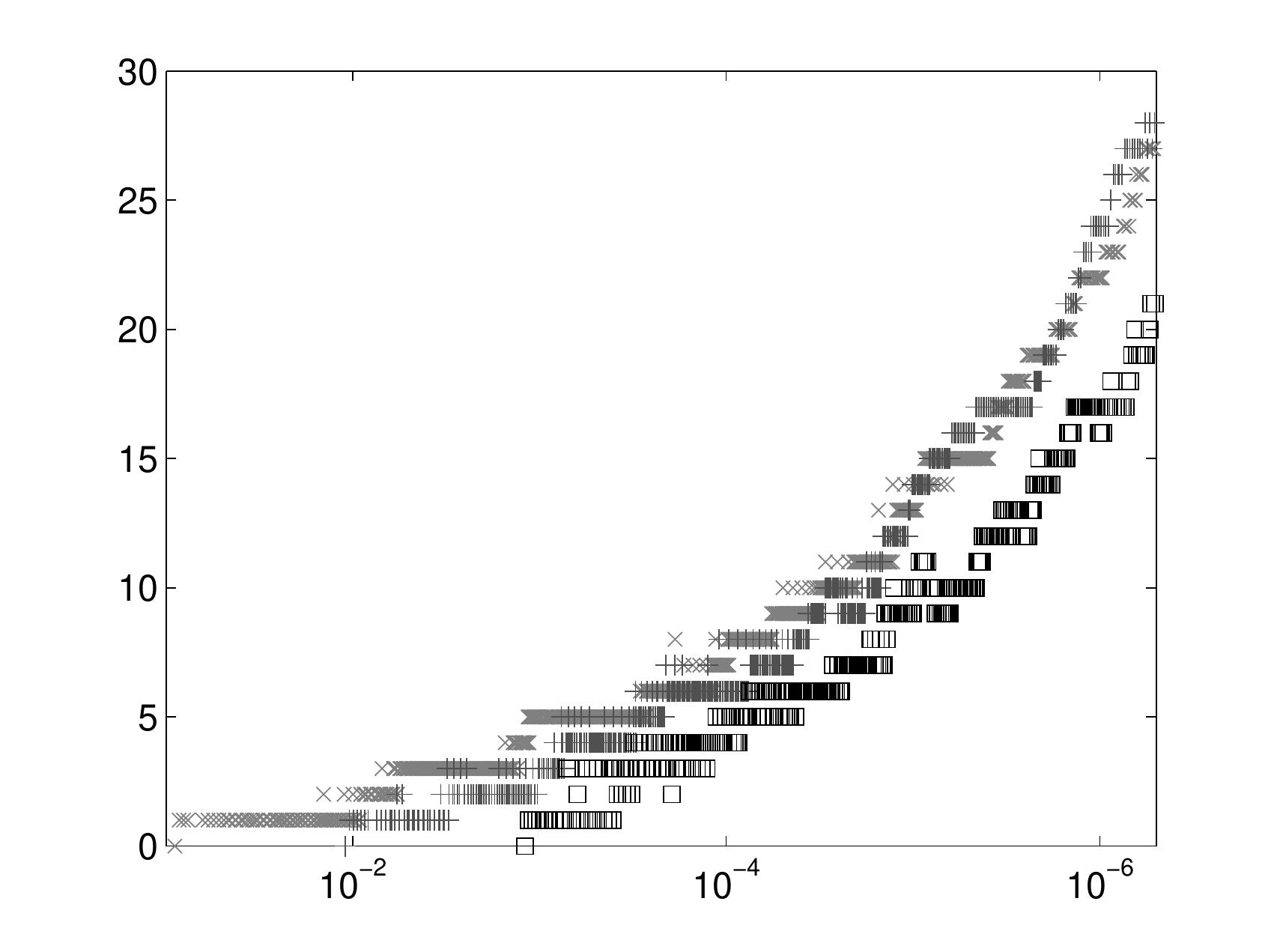} &
\includegraphics[width=6cm]{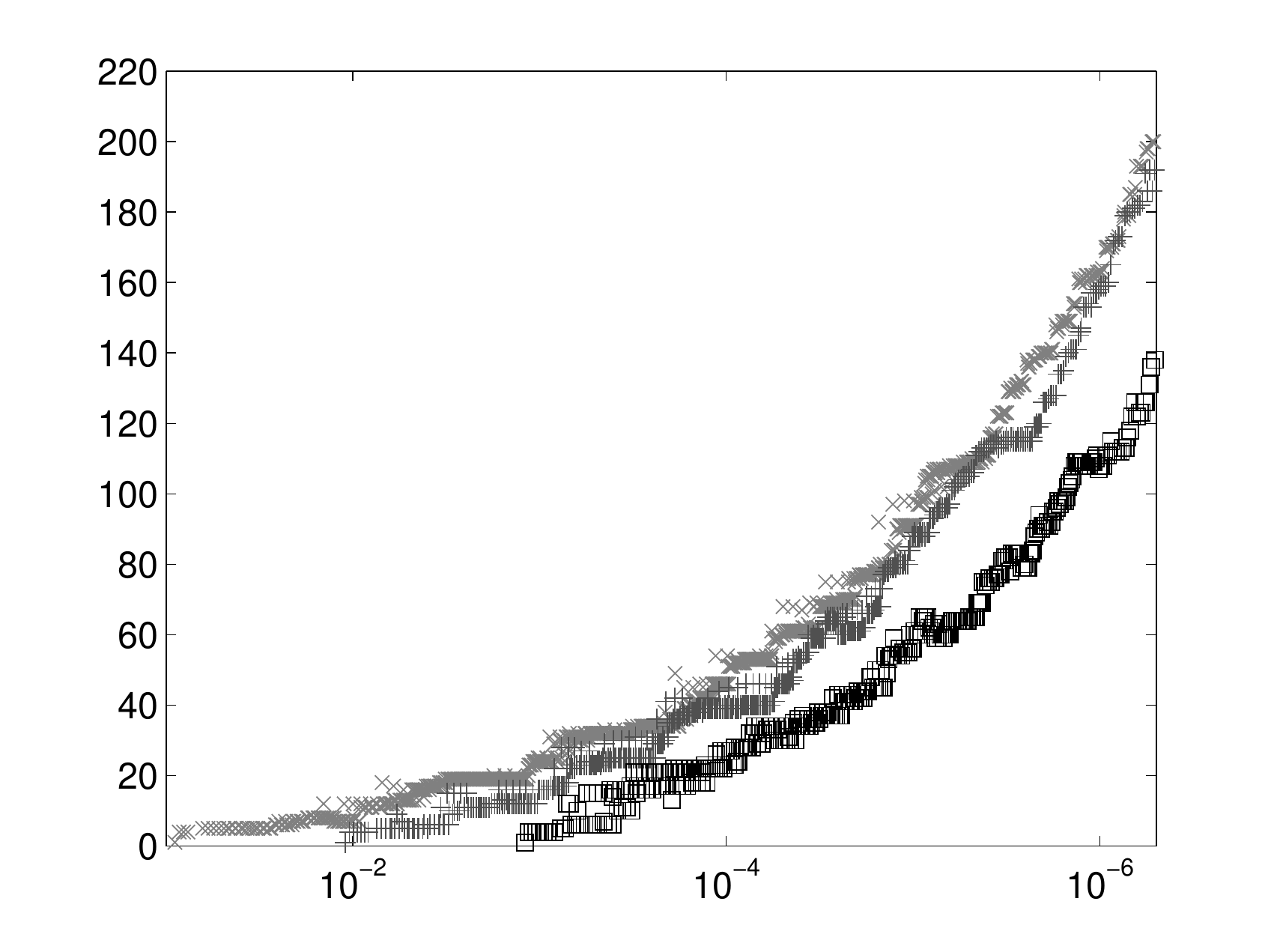}
\end{tabular}
\caption{Tridiagonal diffusion matrix, $a=\frac12$: $\abs{\rank(\bw_{k,j})}_\infty$ (left) and maximum ranks of all intermediates arising in the inner iteration steps (right),
in dependence on current estimate for $\norm{\bu - \bw_{k,j}}$ (horizontal axis), for  $d=\textcolor{lightgray}{\boldsymbol{\times}} 4,\textcolor{gray}{\boldsymbol{+}} 16, \square 64$.}
\label{fig:triranks}
\end{figure}

\begin{figure}[tp]
\centering
\includegraphics[width=8cm]{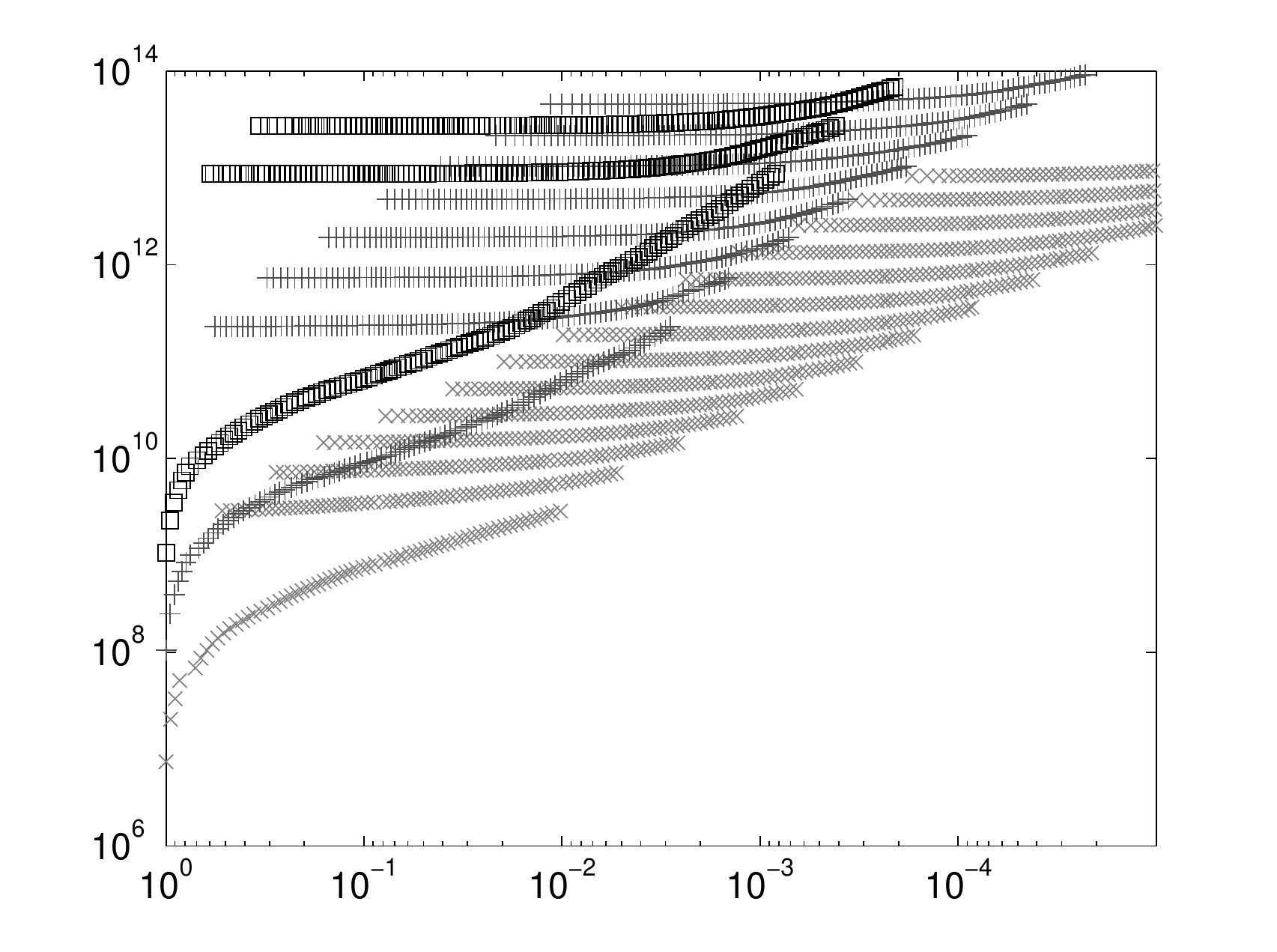}
\caption{Tridiagonal diffusion matrix, $a=\frac12$: operation count in dependence on the error estimate reduction (horizontal axis), for $d=\textcolor{lightgray}{\boldsymbol{\times}} 4,\textcolor{gray}{\boldsymbol{+}} 16, \square 64$.}
\label{fig:triops}
\end{figure}

In summary, we conclude that if convergence to the exact solution is required only in $\spL{2}$, the seeming drawback of losing symmetry in the preconditioned system is more than compensated by the practical simplifications and by the gain in computational efficiency.

\paragraph{Acknowledgements.} The authors would like to thank Kolja Brix for providing multiwavelet construction data used in the numerical experiments.

\bibliography{precondalr}

\end{document}